\documentclass[review,hidelinks,onefignum,onetabnum]{siamart220329}

\hypersetup{
  pdftitle={Tight Probability Bounds with Pairwise Independence},
  pdfauthor={A. K. Ramachandra, and K. Natarajan}
}

\usepackage{lipsum}
\usepackage{amsfonts}
\usepackage{graphicx}
\usepackage{epstopdf}
\usepackage{algorithmic}
\usepackage{enumitem}

\newcommand{\mb}[1]{\mbox{\boldmath $#1$}}

\newcommand{\mbs}[1]{{\mbox{\boldmath \scriptsize{$#1$}}}}

\def\texitem#1{\par\smallskip\noindent\hangindent 25pt
               \hbox to 25pt {\hss #1 ~}\ignorespaces}
          \usepackage{multirow}
          \usepackage{mathtools}
          \usepackage{tikz}
\usetikzlibrary{shapes,arrows,positioning, intersections}

\tikzstyle{block} = [rectangle, draw, fill=blue!20,
    text width=8 em, text centered, rounded corners, minimum height=4em]
\tikzstyle{line} = [draw, -latex']
\newsiamremark{example}{Example}
\newsiamremark{prop}{Proposition}
\usepackage[font=scriptsize]{subcaption}
\DeclarePairedDelimiter{\ceil}{\lceil}{\rceil}

\headers{Tight Probability Bounds with Pairwise Independence}{A. K. Ramachandra, and K. Natarajan}
\begin{document}
\title{Tight Probability Bounds with Pairwise Independence\thanks{Submitted: March 2021, Revised: April 2022.
\funding{The research of the second author was partly supported by the MOE Academic Research Fund Tier 2 grant MOE2019-T2-2-138, ``Enhancing Robustness of Networks to Dependence via Optimization''.}}}

\date{Submitted: March 2021, Revised: April 2022}

\author{Arjun Kodagehalli Ramachandra\thanks{Engineering Systems and Design, Singapore University of Technology and Design, 8 Somapah Road, Singapore 487372. \email{arjun\_ramachandra@sutd.edu.sg}}. \and Karthik Natarajan\thanks{Engineering Systems and Design, Singapore University of Technology and Design, 8 Somapah Road, Singapore 487372. \email{karthik\_natarajan@sutd.edu.sg}}}

\nolinenumbers
\maketitle

\begin{abstract}
While useful probability bounds for $n$ pairwise independent Bernoulli random variables adding up to at least an integer $k$ have been proposed in the literature, none of these bounds are tight in general. In this paper, we provide several results in this direction. Firstly, when $k = 1$, the tightest upper bound on the probability of the union of $n$ pairwise independent events is provided in closed-form for any input marginal probability vector $\mb{p} \in [0,1]^n$.  To prove the result, we show the existence of a positively correlated Bernoulli random vector with transformed bivariate probabilities, which is of independent interest. Building on this, we show that the ratio of the Boole union bound and the tight pairwise independent bound is upper bounded by $4/3$ and that the ratio is attained. Applications of the result in correlation gap analysis and distributionally robust bottleneck optimization are discussed. The result is extended to find the tightest lower bound on the probability of the intersection of $n$ pairwise independent events.  Secondly, for any $k \geq 2$ and input marginal probability vector $\mb{p} \in [0,1]^n$, new upper bounds are derived by exploiting ordering of probabilities.  Numerical examples are provided to illustrate when the bounds provide improvement over existing bounds. Lastly, we identify specific instances when the existing and the new bounds are tight, for example, with identical marginal probabilities.
\end{abstract}



\section{Introduction}  \label{sec:intro}
Probability bounds for sums of Bernoulli random variables have been extensively studied by researchers in various communities including probability and statistics, computer science, combinatorics and optimization. In this paper, our focus is on pairwise independent Bernoulli random variables. It is well known that while mutually independent random variables are pairwise independent, the reverse is not true. Feller \cite{feller} attributes Bernstein \cite{bernstein1946} with identifying one of the earliest examples of $n = 3$ pairwise independent random variables that are not mutually independent. For general $n$, constructions of pairwise independent Bernoulli random variables can be found in the works of Geisser and Mantel \cite{geisser1962}, Karloff and Mansour \cite{karloffmansour1994}, Koller and Meggido \cite{koller}, pairwise independent discrete random variables in Feller \cite{feller1959}, Lancaster \cite{lancaster1965}, Joffe \cite{joffe1974}, O'Brien \cite{brien1980} and pairwise independent normal random variables in Geisser and Mantel \cite{geisser1962}. One of the motivations for studying constructions of pairwise independent random variables particularly in the computer science community is that the joint distribution can have a low cardinality support (polynomial in the number of random variables) in comparison to mutually independent random variables (exponential in the number of random variables). The reader is referred to Lancaster \cite{lancaster1965} and more recent papers of Babai \cite{Babai} and Gavinsky and Pudl\'{a}k \cite{Gavinsky} who provide precise lower bounds on the entropy of the joint distribution of pairwise independent random variables that only grow logarithmically with the number of random variables. The low cardinality of such distributions have important ramifications in the efficient derandomization of algorithms for NP-hard combinatorial optimization problems (see the review article of Luby and Widgerson \cite{luby} and the references therein for results on pairwise independent and more generally $t$-wise independent random variables).

In this paper, we are interested in the problem of computing probability bounds for the sum of pairwise independent Bernoulli random variables adding up to at least an integer $k$. Given an integer $n \geq 2$, denote by $[n]=\{1,2,\ldots,n\}$ and by $K_n=\{(i,j): 1 \leq i < j \leq n\}$ (it can be viewed as a complete graph on $n$ nodes). Given integers $i < j$, let $[i,j] = \{i,i+1,\ldots,j-1,j\}$. Consider a Bernoulli random vector $\tilde{\mb{c}} = (\tilde{c}_{1},\ldots,\tilde{c}_n)$ with marginal probabilities given by $p_i = \mathbb{P}(\tilde{c}_{i} = 1)$ for $i \in [n]$. Denote by $\mb{p} = (p_1,\ldots,p_n) \in [0,1]^n$, the univariate marginal vector and by $\Theta(\{0,1\}^n)$, the set of all probability distributions supported on $\{0,1\}^n$. Consider the set of joint probability distributions of Bernoulli random variables consistent with the given marginal probabilities and pairwise independence:
\begin{equation*}
\begin{array}{rlll}
\displaystyle \Theta(\mb{p},p_ip_j; (i,j) \in K_n) = \Big\{\theta \in \Theta(\{0,1\}^n) \ \Big{|} \ \mathbb{P}_\theta\left(\tilde{c}_i = 1
    \right) = p_i, \forall
    i \in [n],\;\;  \\
 \displaystyle   \mathbb{P}_\theta\left(\tilde {c}_i = 1, \tilde {c}_j= 1
    \right) = p_{i}p_{j}, \;\forall (i,j) \in K_n\Big\}.
 \end{array}
 \end{equation*}
This set of distributions is nonempty for any $\mb{p} \in [0,1]^n$, since the distribution of mutually independent random variables lies in the set. Our problem of interest is to compute the maximum probability that $n$ random variables adds up to at least  an integer $k \in [n]$ over all distributions in the set. Denote this tightest upper bound by $\overline{P}(n,k,\mb{p})$ (observe that the bivariate probabilities here are simply given by the product of the univariate probabilities). Then,
\begin{equation}\label{eq:p-indprobdefn}
\begin{array}{lll}
\displaystyle \overline{P}(n,k,\mb{p})  =  \displaystyle \max_{\theta \in \Theta(\mbs{p},p_ip_j; (i,j) \in K_n)} \mathbb{P}_\theta\left( \sum_{i \in [n]} \tilde{c}_{i} \geq k\right).
\end{array}
\end{equation}
Two useful bounds that have been proposed for this problem are discussed next:
\texitem{(a)} Chebyshev \cite{chebyshev1867} bound: The one-sided version of the Chebyshev tail probability bound uses the first and second moments of the random variables. Since the Bernoulli random variables are assumed to be pairwise independent or equivalently uncorrelated, the variance of the sum is given by:
$$\displaystyle \mbox{Variance}\left( \sum_{i \in [n]} \tilde{c}_{i}\right) = \sum_{i \in [n]} p_i(1-p_i).$$
Applying the Chebyshev bound gives:
\begin{equation}\label{cheby}
\begin{array}{lll}
\overline{P}(n,k,\mb{p}) \leq \begin{cases}
 1,   & k <  {\displaystyle \sum_{i \in [n]} p_{i}}, \\
  \displaystyle \frac{\sum_{i \in [n]} p_{i}(1-p_{i})}{\sum_{i \in [n]} p_{i}(1-p_{i}) +( k-  \sum_{i \in [n]} p_{i} )^2} ,&{ \displaystyle \sum_{i \in [n]} p_{i}} \leq k \leq n.
  \end{cases}
\end{array}
\end{equation}

\texitem{(b)} Schmidt, Siegel and Srinivasan \cite{shrinivasan1995} bound: The Schmidt, Siegel and Srinivasan bound is derived by bounding the tail probability using the moments of multilinear polynomials. This is in contrast to the Chernoff-Hoeffding bound (see Chernoff \cite{chernoff}, Hoeffding \cite{hoeffding}) which bounds the tail probability of the sum of independent random variables using the moment generating function. A multilinear polynomial of degree $j$ in $n$ variables is defined as:
$$\displaystyle S_{j}(\mb{c})= \sum_{1\leq i_{1}< i_{2}< \ldots < i_{j}\leq n} {c}_{i_{1}} {c}_{i_{2}} \ldots {c}_{i_{j}}.$$
At the crux of the analysis in \cite{shrinivasan1995} is the observation that all the higher moments of the sum of Bernoulli random variables can be generated using linear combinations of the expected values of multilinear polynomials of the random variables. The construction of the bound makes use of the equality:
\begin{equation}
\begin{array}{lll}   \label{binomial}
\displaystyle \binom {\sum_{i\in [n]}{c}_i}{j}   =  \displaystyle S_{j}(\mb{c}), & \forall \mb{c} \in \{0,1\}^n, \forall j \in [0,\sum_{i\in [n]}{c}_i],
\end{array}
\end{equation}
where $S_0(\mb{c}) = 1$ and $\binom {r}{s} = r!/(s!(r-s)!)$ for any pair of integers $r \geq s \geq 0$. The bound derived in Schmidt et al. \cite{shrinivasan1995} (see Theorem 7, part (II) on page 239) for pairwise independent random variables is\footnote{While the statement in the theorem in \cite{shrinivasan1995} is provided for $k > \sum_{i} p_i$, it is straightforward to see that their analysis would lead to the form provided here for general $k$.}:
\begin{equation}
\begin{array}{lll}   \label{SSS}
\overline{P}(n,k,\mb{p}) \leq
 \min\bigg(1, \displaystyle\frac{\sum_{i \in [n]} p_{i}}{k} ,\frac{\sum_{(i,j) \in K_n}p_{i}p_{j}}{\binom{k}{2}}\bigg).
\end{array}
\end{equation}

While both the Chebyshev bound in \eqref{cheby} and the Schmidt, Siegel and Srinivasan bound in \eqref{SSS} are useful, neither of them are tight for general values of $n$, $k$ and $\mb{p} \in [0,1]^n$. In this paper, we work towards tightening these bounds for pairwise independent random variables and identifying instances when the bounds are tight.

\subsection{Other related bounds}
Consider the set of joint distributions of Bernoulli random variables consistent with the marginal probability vector $\mb{p} \in [0,1]^n$ and general bivariate probabilities given by $p_{ij} = \mathbb{P}(\tilde{c}_{i} = 1,\tilde{c}_{j} = 1)$ for all $(i,j) \in K_n$:
\begin{equation*}
\begin{array}{rlll}
\displaystyle \Theta(\mb{p},p_{ij}; (i,j) \in K_n) = \Big\{\theta \in \Theta(\{0,1\}^n) \ \Big{|} \ \mathbb{P}_\theta\left(\tilde{c}_i = 1
    \right) = p_i, \forall
    i \in [n],\;\;  \\
 \displaystyle   \mathbb{P}_\theta\left(\tilde {c}_i = 1, \tilde {c}_j= 1
    \right) = p_{ij}, \;\forall (i,j) \in K_n\Big\}.
\end{array}
\end{equation*}
Unlike the pairwise independent case, verifying if this set of distributions is nonempty is already known to be a NP-complete problem (see Pitowsky \cite{pitowsky91}). The tightest upper bound on the tail probability over all distributions in this set is given by:
\begin{equation*}
\begin{array}{lll}
\displaystyle \max_{\theta \in \Theta(\mbs{p},p_{ij}; (i,j) \in K_n)} \mathbb{P}_\theta\left( \sum_{i \in [n]} \tilde{c}_{i} \geq k\right),
\end{array}
\end{equation*}
where the bound is set to $-\infty$ if the set of feasible distributions is empty. The bound is given by the optimal value of the linear program (see Hailperin \cite{hailperin}):
\begin{equation} \label{eq:generalprimal}
\begin{array}{rlllll}
\displaystyle \max & \displaystyle \sum_{\mbs{c} \in \{0,1\}^n: \sum_{t} {c}_{t} \geq k} \theta(\mb{c})\\
 \mbox{s.t}  &\displaystyle  \sum_{\mbs{c} \in \{0,1\}^n} \theta(\mb{c})= 1, \\
 & \displaystyle \sum_{\mbs{c} \in \{0,1\}^n: c_i = 1} \theta(\mb{c})= p_i , & \forall i \in [n],\\
 & \displaystyle \sum_{\mbs{c} \in \{0,1\}^n: c_i = 1, c_j = 1}\theta(\mb{c})= p_{ij}, & \forall (i,j) \in K_n, \\
  & \displaystyle \theta(\mb{c}) \geq 0, & \forall \mb{c} \in \{0,1\}^n,
\end{array}
\end{equation}
where the decision variables are the joint probabilities $\theta(\mb{c}) = \mathbb{P}(\tilde{\mb{c}} = \mb{c})$ for all $\mb{c} \in \{0,1\}^n$. The number of decision variables in the formulation grows exponentially in the number of random variables $n$. The dual linear program is given by:
\begin{equation} \label{eq:generaldual}
\begin{array}{rlllll}
\displaystyle \min & \displaystyle \sum_{(i,j) \in K_n} \lambda_{ij}p_{ij}+ \displaystyle \sum_{i \in [n]} \lambda_{i}p_{i} +\lambda_{0} \\
 \mbox{s.t}  &\displaystyle \sum_{(i,j) \in K_n} \lambda_{ij}{c}_{i}{c}_{j}+ \displaystyle \sum_{i \in [n]} \lambda_{i}{c}_{i} +\lambda_{0} \geq 0, & \forall \mb{c} \in \{0,1\}^n,\\
 &\displaystyle \sum_{(i,j) \in K_n} \lambda_{ij}{c}_{i}{c}_{j}+ \displaystyle \sum_{i \in [n]} \lambda_{i}{c}_{i} +\lambda_{0} \geq 1, & \forall \mb{c} \in \{0,1\}^n: \sum_{t} {c}_{t} \geq k.
\end{array}
\end{equation}
The dual linear program in \eqref{eq:generaldual} has a polynomial number of decision variables but an exponential number of constraints. This linear program is always feasible (simply set $\lambda_0=1$ and remaining dual variables to be zero) and strong duality thus holds. Given the large size of the primal and dual linear programs that need to be solved, two main approaches have been studied in the literature:
\texitem{(a)} The first approach is to find closed-form bounds by generating simple dual feasible solutions (see Kounias \cite{kounias1968}, Kounias and Marin \cite{kounias1976}, Sathe et al. \cite{sathe1980}, M\'{o}ri and Sz\'{e}kely \cite{mori1985}, Dawson and Sankoff \cite{dawson1967}, Galambos \cite{galambos1975,galambos1977}, de Caen \cite{decaen1997}, Kuai et al. \cite{tahakara2000}, Dohmen and Tittmann \cite{dohmen2007} and related graph-based bounds in Hunter \cite{hunter1976}, Worsley \cite{worsley1982}, Veneziani \cite{veneziani2008union}, Vizv\'{a}ri \cite{vizvari2007}). These bounds have shown to be tight in specific instances (in Section \ref{subsec:hunter} we discuss some of these instances).
\texitem{(b)} The second approach is to reduce the size of the linear programs used and solve them numerically. As the number of random variables $n$ increase, the linear programs quickly become intractable and thus many papers adopting this approach, aggregate the primal decision variables, thus obtaining weaker bounds as a trade-off for the reduced size. Formulations of linear programs using partially or fully aggregated univariate, bivariate or $m$-variate information for $2 \leq m < n$ have been proposed in Kwerel \cite{kwerelstringent1975}, Platz \cite{platz1985}, Pr\'{e}kopa \cite{prekopa1988,prekopa1990}, Boros and Pr\'{e}kopa \cite{boros1989}, Pr\'{e}kopa and Gao \cite{prekopa2005}, Qiu et al. \cite{qiu2016}, Yang et al. \cite{yang2016}, Yoda and Pr\'{e}kopa \cite{yoda2016}). Techniques to solve the dual formulation have been studied in Boros et al. \cite{boros2014}.

Using the second approach, in some cases, closed-form bounds have been derived as solutions of the aggregated linear programs. One such bound which is of relevance to this paper is developed in Boros and Pr\'{e}kopa \cite{boros1989} when the first and second binomial moments of an integer random variable supported on $[0,n]$ are known. They computed the tightest upper bound on $ \mathbb{P}(\tilde{\xi} \geq k)$ by considering all distributions $\omega$ of an integer random variable $\tilde{\xi}$ supported on $[0,n]$ given by the set:
\begin{equation*}
\begin{array}{lll}
\displaystyle \left\{\omega([0,n]) \ \Big{|} \ \mathbb{E}_{\omega}\left[\displaystyle \binom {\;\tilde{\xi}\;}{j}
    \right] = S_{j}, \; j=1,2\right\}.
\end{array}
\end{equation*}
Setting $\tilde{\xi} = \sum_i \tilde{c}_i$ with $S_{1} = \mathbb{E}[S_{1}(\tilde{\mb{c}})]$ and $S_{2} = \mathbb{E}[S_{2}(\tilde{\mb{c}})]$ gives a closed-form upper bound as follows:
\begin{equation} \label{BorosPrekopa}
\begin{array}{lll}
\displaystyle \mathbb{P}\left( \sum_{i \in [n]} \tilde{c}_{i} \geq k\right)  \leq \begin{cases}
 1,   &  k< \displaystyle\frac{(n-1)S_{1}-2S_{2}}{n-S_{1}},\\
\displaystyle \frac{(k+n-1)S_{1}-2S_{2}}{kn} ,&\displaystyle \frac{(n-1)S_{1}-2S_{2}}{n-S_{1}} \leq k < 1+\displaystyle \frac{2S_{2}}{S_{1}}, \\
\displaystyle \frac{(i-1)(i-2S_{1})+2S_{2}}{(k-i)^2+(k-i)} ,& k \geq 1+\displaystyle\frac{2S_{2}}{S_{1}},
\end{cases}
\end{array}
\end{equation}
where $i=\lceil({(k-1)S_{1}-2S_{2}})/({k-S_{1}})\rceil$ and the ceiling function $\lceil x \rceil$ maps $x$ to the smallest integer greater than or equal to $x$. Similar to the Chebyshev bound and the Schmidt, Siegel and Srinivasan bound, the Boros and Pr\'{e}kopa bound in \eqref{BorosPrekopa} is not generally tight since it uses aggregated moment information, rather than the specific marginal probabilities. Another useful upper bound derived under weaker assumptions is the Boole union bound \cite{boole1854} (see also Fr\'{e}chet \cite{frechet1935}) for $k = 1$. This bound is valid even with arbitrary dependence among the Bernoulli random variables. Let $\Theta(\mb{p})$ denotes the set of joint distributions supported on $\{0,1\}^n$ consistent with the univariate marginal probability vector $\mb{p} \in [0,1]^n$. The Boole union bound is given as:
\begin{equation}\label{frechet}
\begin{array}{rlllll}
\displaystyle {\overline{P}}_u(n,1,\mb{p})  =  \displaystyle \max_{\theta \in \Theta(\mbs{p})} \mathbb{P}_\theta\left( \sum_{i \in [n]} \tilde{c}_{i} \geq 1\right)   = \displaystyle \min\left(\sum_{i \in [n]} p_{i},1\right).
\end{array}
\end{equation}
Clearly, $\overline{P}(n,1,\mb{p}) \leq {\overline{P}}_u(n,1,\mb{p})$. Extensions of this bound for $k \geq 2$ is provided in R\"{u}ger \cite{ruger1978}.

\subsection{Contributions and structure}
This brings us to the key contributions and the structure of the current paper:
\texitem{(a)} In Section \ref{sec:unionbound}, we establish (see Lemma \ref{lem:bivarfeas}) that a positively correlated Bernoulli random vector $\tilde{\mb{c}}$ with the univariate probability vector $\mb{p} \in [0,1]^n$ and transformed bivariate probabilities $p_ip_j/p$ where $\max_i p_i \leq p \leq 1$, always exists. The lemma helps us compute the tightest upper bound on the probability of the union of $n$ pairwise independent events and is of independent interest. By a simple transformation, the results from Lemma \ref{lem:bivarfeas} are extended to show the existence of an alternate positively correlated Bernoulli random vector (see Corollary \ref{cor:bivarfeasvariant}). Feasibility is not guaranteed for arbitrary correlation structures with Bernoulli random vectors and hence these two results provide useful sufficient conditions.
\texitem{(b)}
We then provide the tightest upper bound on the probability on the union of $n$ pairwise independent events ($k=1$) in closed-form (see Theorem \ref{thm:unionbound}). The contributions of Theorem \ref{thm:unionbound} lie in:
\begin{enumerate}
\item
Establishing that when the random variables are pairwise independent, for any given marginal vector $\mb{p} \in [0,1]^n$,  the upper bound proposed in Kounias \cite{kounias1968}, Hunter \cite{hunter1976} and Worsley \cite{worsley1982} is tight. These bounds were initially developed for the sum of dependent Bernoulli random variables with arbitrary bivariate probabilities (using tree structures from graph theory) and are not tight in general (see Example \ref{example:hw} in Section \ref{subsec:hunter}). Interestingly for pairwise independent random variables, we prove that the bound is tight by using Lemma \ref{lem:bivarfeas}.
\item Providing an explicit construction of an extremal distribution (not unique) that attains this bound (see Table \ref{table:hwprobdist}).
\item Proving that the ratio of the Boole union bound and the pairwise independent bound is upper bounded by $4/3$ and that this is attained (see Proposition \ref{prop:25}). Applications of the result in correlation gap analysis and distributionally robust bottleneck combinatorial optimization are discussed (see examples \ref{ex:correlationanalysis} and \ref{ex:bottleneck}).
\item
Deriving the tightest lower bound on the probability of the intersection of $n$ pairwise independent events ($k=n$) in closed-form (see Corollary \ref{cor:intersection}).
\end{enumerate}
\texitem{(c)} In Section \ref{sec:nonidenticalnewbounds}, we focus on  $k \geq 2$ and present new bounds exploiting the ordering of probabilities (see Theorem \ref{thm:nonidentical}). These ordered bounds improve on the closed-form bounds discussed in Section \ref{sec:intro} and numerical examples are provided to illustrate this result.
\texitem{(d)} In Section \ref{sec:tightinstances}, we provide instances where some of the existing bounds and the newly proposed ordered bounds are tight:
\begin{enumerate}
     \item First, we identify a special case when the existing closed-form bounds are tight. When the random variables are identically distributed, in Section \ref{subsec:idenBP}, we provide the tightest upper bound in closed-form  (see Theorem \ref{thm:idenBPtight}) for any $k \in [n]$. The proof is based on showing an equivalence with a linear programming formulation of an aggregated moment bound for which closed-form solutions have been derived by Boros and Pr\'{e}kopa \cite{boros1989}. While the expression of the tight closed-form bound is complicated in form in comparison with the Chebyshev bound in \eqref{cheby} and the Schmidt, Siegel and Srinivasan bound in \eqref{SSS}, it helps us identify conditions when the latter bounds are guaranteed to be tight (see Proposition \ref{prop:tighthom}).
     \item This result with identical marginals is further extended to show tightness for more general $t$-wise independent variables (see Corollary \ref{cor:twiseiden}). The tight bounds for $t \geq 4$ can be derived as the optimal solution to an aggregated linear program first proposed by Pr\'{e}kopa \cite{prekopa1990}.
     \item Next, when $n-1$ marginal probabilities are identical, Proposition \ref{prop:tightnonid} provides instances when the new ordered bounds are tight. Numerical examples are provided to illustrate this result.
\end{enumerate}
\texitem{(e)} We conclude in Section \ref{sec:conclusion} and identify some future research questions.

\section{Tight upper bound for $k = 1$} \label{sec:unionbound}
The goal of this section is to provide the tightest upper bound on the probability of the union of pairwise independent events. Towards this, we start by   generating a feasible solution to the dual linear program in (\ref{eq:generaldual}) with $k = 1$,  $p_{ij} = p_ip_j$ for all $(i,j) \in K_n$ and probabilities sorted in increasing value as $0 \leq p_{1} \leq p_{2} \leq \ldots \leq p_{n} \leq 1$. Consider the dual solution:
\begin{equation*}
\begin{array}{rlllll}
\lambda_{0}= 0, \; \lambda_{i}= 1 \; \forall i \in [n], \; \lambda_{in}= -1 \; \forall i \in [n-1]  \mbox{ and } \lambda_{ij} = 0 \mbox{ otherwise}.
\end{array}
\end{equation*}
The left hand side of the dual constraints in (\ref{eq:generaldual}) then simplifies to:
\begin{equation*}
\begin{array}{rlllll}
\displaystyle \sum_{(i,j) \in K_n} \lambda_{ij}{c}_{i}{c}_{j}+ \sum_{i \in [n]} \lambda_{i}{c}_{i} +\lambda_{0} & =&  \displaystyle -\sum_{i \in [n-1]} {c}_{i}c_n + \sum_{i \in [n]} {c}_{i} \\
& = & \displaystyle c_n + \sum_{i\in [n-1]} {c}_{i}(1-c_n).
\end{array}
\end{equation*}
To verify that this solution is dual feasible, observe that with all $c_i = 0$, $c_n + \sum_{i\in [n-1]} {c}_{i}(1-c_n) = 0$. When $c_n = 1$, regardless of the values of $c_1,\ldots,c_{n-1}$, we have $c_n + \sum_{i \in [n-1]} {c}_{i}(1-c_n) = 1$. Lastly, when $c_n = 0$ and at least one $c_{i} = 1$ for $i \in [n-1]$, we have $c_n + \sum_{i \in [n-1]} {c}_{i}(1-c_n) \geq 1$. This solution has an objective value of $\sum_{i \in [n]} p_{i}-p_{n}(\sum_{i \in [n-1]} p_{i})$. From weak duality and using the trivial upper bound of $1$, we have:
$$\displaystyle \overline{P}(n,1,\mb{p}) \leq \min\left(\sum_{i \in [n]} p_{i}-p_{n}\left(\sum_{i \in [n-1]} p_{i}\right),1\right).$$ Intuitively the first term in this expression is obtained using the probabilistic inequality:
\begin{equation*}
\begin{array}{rlllll}
\displaystyle \mathbb{P}\left( \sum_{i \in [n]} \tilde{c}_{i} \geq 1\right)  \leq \sum_{j \in [n-1]} \mathbb{P}\left( \tilde{c}_{j}=1, \tilde{c}_{n}=0\right)+\mathbb{P}\left( \tilde{c}_{n}=1\right),
\end{array}
\end{equation*} and is provided in the work of Kounias \cite{kounias1968}. The key result we show is that it is always possible to construct a pairwise independent distribution which attains the upper bound. The proof involves showing that the problem can be transformed to proving the existence of a distribution of a Bernoulli random vector $\tilde{\mb{c}}$ with univariate probabilities given by $\mathbb{P}(\tilde{c}_{i} = 1) = p_i$ and transformed bivariate probabilities given by $\mathbb{P}(\tilde{c}_{i} = 1,\tilde{c}_j=1)= p_ip_j/p_n$, where $p_n$ is the largest univariate probability.  In the following lemma, we prove a more general result on the existence of such a correlated Bernoulli random vector.

\begin{lemma}\label{lem:bivarfeas}
Given an arbitrary univariate probability vector $\mb{p} \in [0,1]^n$ and bivariate probabilities $p_ip_j/p$ for $(i,j) \in K_n$ where $\max_i p_i \leq p \leq 1$, a Bernoulli random vector consistent with the given univariate and bivariate probabilities always exists.
\end{lemma}
\begin{proof}
 Sort the probabilities in increasing value as $0 \leq p_{1} \leq p_{2} \leq \ldots \leq p_{n} \leq 1$. We want to show that there always exists a distribution $\theta \in \displaystyle \Theta(\mb{p},p_{i}p_j/p; (i,j) \in K_n)$ such that:
\begin{equation} \label{eq:compslack3aa}
\begin{array}{rlll}
\displaystyle \sum_{\mbs{c} \in \{0,1\}^{n}} \theta(\mb{c}) & = &1, & \\
\displaystyle \sum_{\mbs{c} \in \{0,1\}^{n}: c_i = 1} \theta(\mb{c}) & = & p_i, & \forall i \in [n], \\
\displaystyle \sum_{\mbs{c} \in \{0,1\}^{n}: c_i = 1, c_j = 1} \theta(\mb{c}) & = & \displaystyle \frac{p_ip_j}{p}, & \forall (i,j) \in K_{n},
\end{array}
\end{equation}
where $p_n \leq p \leq 1$. The proof is divided into two parts:
\texitem{(1)} We first argue that it is sufficient to verify the existence of joint probabilities $\theta(\mb{c})$ for $n$ Bernoulli random variables such that:
\begin{equation} \label{eq:compslack3a}
\begin{array}{rlll}
\displaystyle \sum_{\mbs{c} \in \{0,1\}^{n}} \theta(\mb{c}) & = &1, & \\
\displaystyle \sum_{\mbs{c} \in \{0,1\}^{n}: c_i = 1} \theta(\mb{c}) & = & p_i, & \forall i \in [n], \\
\displaystyle \sum_{\mbs{c} \in \{0,1\}^{n}: c_i = 1, c_j = 1} \theta(\mb{c}) & = & \displaystyle \frac{p_ip_j}{p_{n}}, & \forall (i,j) \in K_{n},
\end{array}
\end{equation}
where the bivariate probabilities are modified from $p_ip_j/p$ to $p_ip_j/p_{n}$. This is because with $1 \leq 1/p \leq 1/p_{n}$, we can find a $\lambda \in [0,1]$ such that:
$$\displaystyle \frac{1}{p} = \lambda \frac{1}{p_{n}} +(1-\lambda)1.$$
Then, we can create the convex combination of two distributions $\overline{{\theta}}$ and $\underline{{\theta}}$ as follows:
$$\displaystyle {\theta} = \lambda \overline{{\theta}} +(1-\lambda)\underline{{\theta}},$$
where $\overline{{\theta}}$ is a probability distribution which satisfies (\ref{eq:compslack3a}) and $\underline{{\theta}}$ is a pairwise independent joint distribution on $n$ Bernoulli random variables with univariate probabilities given by $p_i$ and bivariate probabilities given by $p_ip_j$. The distribution $\underline{{\theta}}$ always exists as we can simply choose the mutually independent distribution on $n$ random variables with univariate probabilities $p_i$. The convex combination then guarantees the existence of a distribution ${\theta}$ which satisfies \eqref{eq:compslack3aa}. In step (2), we prove the existence of such a $\overline{{\theta}}$.
\texitem{(2)} To show that \eqref{eq:compslack3a} is feasible, observe that there always exists a feasible distribution on $n-1$ Bernoulli random variables with probabilities given by $\vartheta(\mb{c}_{-n}) = \mathbb{P}(\tilde{\mb{c}}_{-n} = \mb{c}_{-n})$ for all $\mb{c}_{-n} = (c_1,\ldots,c_{n-1}) \in \{0,1\}^{n-1}$ such that:
\begin{equation} \label{eq:compslack4}
\begin{array}{rlll}
\displaystyle \sum_{\mbs{c}_{-n} \in \{0,1\}^{n-1}} \vartheta(\mb{c}_{-n}) & = &1, & \\
\displaystyle \sum_{\mbs{c}_{-n} \in \{0,1\}^{n-1}: c_i = 1} \vartheta(\mb{c}_{-n}) & = & \displaystyle \frac{p_i}{p_{n}}, & \forall i \in [n-1], \\
\displaystyle \sum_{\mbs{c}_{-n} \in \{0,1\}^{n-1}: c_i = 1, c_j = 1} \vartheta(\mb{c}_{-n}) & = & \displaystyle \frac{p_ip_j}{p_{n}^2}, & \forall (i,j) \in K_{n-1}.
\end{array}
\end{equation}
Such a ${\vartheta}$ exists because we can simply choose the mutually independent distribution on $n-1$ random variables with univariate probabilities $p_i/p_{n}$ where the bivariate probabilities are given by $(p_i/p_n)(p_j/p_n)$. Then, we construct the distribution on $n$ random variables by setting the probability of the vector of all zeros to $1-p_{n}$, setting the probabilities of the scenarios $\mathbb{P}(\tilde{\mb{c}}_{-n} = \mb{c}_{-n},\tilde{{c}}_{n}=1)$ to $\vartheta(\mb{c}_{-n})p_n$ and setting all the remaining probabilities to zero. This creates a feasible distribution satisfying \eqref{eq:compslack3a} as seen in the construction of Table \ref{table:probaa}. This completes the proof.
\begin{table}[htbp]
\footnotesize
\caption{Probabilities of the scenarios to create a feasible distribution $\overline{{\theta}}$ in (\ref{eq:compslack3a}).}\label{table:probaa}
\begin{center}
\begin{tabular}{lllllll}
\mbox{Scenarios} & $c_1$ & $c_2$ & \ldots & $c_{n}$ &  \mbox{Probability} \\ \hline
\multirow{4}*{$2^{n-1}$$\begin{dcases*} \\ \\ \\ \end{dcases*}$} & 0 & 0 & \ldots & 0 & $\theta(\mb{c}) = 1-p_{n}$ \\
& 1 & 0 & \ldots & 0  & $0$   \\
& \vdots & \vdots & \vdots & \vdots  &  \vdots  \\
& 1 &  1 & \ldots &0 & $0$  \\
  \multirow{3}*{$2^{n-1}$$\begin{dcases*} \\ \\ \\ \end{dcases*}$}& 0 & 0 & \ldots &1 &  $\theta(\mb{c}) = p_{n}\vartheta(\mb{c}_{-n})$ \\
& \vdots & \vdots & \vdots & \vdots &  \vdots  \\
& 1& 1 & \ldots & 1 &  $\theta(\mb{c}) = p_{n}\vartheta(\mb{c}_{-n})$ \\
\end{tabular}
\end{center}
\end{table}
\end{proof}
We remark that there are alternative approaches to construct distributions satisfying Lemma \ref{lem:bivarfeas}. An anonymous referee provided the following construction. Let $\tilde{\mb{d}}$ denote a Bernoulli random vector with mutually independent random variables with marginal probabilities given by $\mathbb{P}(\tilde{d}_{i} = 1) = p_i/p$ for $i \in [n]$ and a Bernoulli random variable $\tilde{z}$ constructed independently with $\mathbb{P}(\tilde{z} = 1) = p$. Define $\tilde{c}_i = \tilde{d}_i\tilde{z}$ for $i \in [n]$. Then $\mathbb{P}(\tilde{c}_{i} = 1) = p_i$ for $i \in [n]$ and $\mathbb{P}(\tilde{c}_{i} = 1,\tilde{c}_j = 1) = p_ip_j/p$ for $(i,j) \in K_n$. We next show that Lemma \ref{lem:bivarfeas} can be extended to prove the existence of an alternative positively correlated Bernoulli random vector.
\begin{corollary}\label{cor:bivarfeasvariant}
Given an arbitrary univariate probability vector $\mb{p} \in [0,1]^n$ and bivariate probabilities $p_ip_j+\frac{p}{1-p}(1-p_i)(1-p_j)$ for $(i,j) \in K_n$ where $0 \leq p \leq\min_i p_i$, a Bernoulli random vector consistent with the given univariate and bivariate probabilities always exists.
\end{corollary}
\begin{proof}
From Lemma \ref{lem:bivarfeas}, it is straightforward to see that there exists a feasible bivariate distribution $\vartheta$ with univariate probabilities $1-p_i$ and bivariate probabilities $(1-p_i)(1-p_j)/(1-p)$ where $0 \leq p \leq\min_i p_i$ (since $1 \geq 1-p\geq \max_i (1-p_i)$). Note that this distribution satisfies
 $\mathbb{P}_{\vartheta}\left(\tilde{\mb{c}_i}=0\right)=p_i, \;  \forall i \in [n]$ and
 \begin{equation*}
\begin{array}{lllll}
\mathbb{P}_{\vartheta}\left(\tilde{\mb{c}_i}=0,\tilde{\mb{c}_j}=0\right)& =&\mathbb{P}_{\vartheta}\left(\tilde{\mb{c}_i}=0\right)-\left[\mathbb{P}_{\vartheta}\left(\tilde{\mb{c}_j}=1\right)-\mathbb{P}_{\vartheta}\left(\tilde{\mb{c}_i}=1,\tilde{\mb{c}_j}=1\right)\right]\\
& =&p_i-\left[(1-p_j)+(1-p_i)(1-p_j)/(1-p)\right]\\
& = & p_ip_j+\frac{p}{1-p}(1-p_i)(1-p_j),
\end{array}
\end{equation*}
for all $(i,j) \in K_{n}$.
By flipping the zeros and ones of the support of $\vartheta$ while retaining the same joint probabilities $\vartheta (\mb{c})$, we obtain the desired result.
\end{proof}

We note that Lemma \ref{lem:bivarfeas} and Corollary \ref{cor:bivarfeasvariant} provide conditions on the bivariate probabilities which guarantee the feasibility of positively correlated Bernoulli random vectors. Feasibility is typically not guaranteed for arbitrary correlation structures with Bernoulli random vectors. While prior works have identified specific correlation structures that are compatible with Bernoulli random vectors (see Chaganty and Joe \cite{Rao}, Qaqish \cite{Qaqish}, Emrich and Piedmonte \cite{Emrich}, Lunn and Davies \cite{Lunn}), the identified conditions in Lemma \ref{lem:bivarfeas} and Corollary \ref{cor:bivarfeasvariant} appear to be new to the best of our knowledge. This brings us to the first theorem, which provides the tightest upper bound on the probability of the union of $n$ pairwise independent events using Lemma \ref{lem:bivarfeas}.

\begin{theorem}\label{thm:unionbound}
Sort the probabilities in increasing value as $0 \leq p_{1} \leq p_{2} \leq \ldots \leq p_{n} \leq 1$. Then,
\begin{equation}
 \label{union1}
 \begin{array}{rlllll}
\displaystyle \overline{P}(n,1,\mb{p})  =  \displaystyle \min\left(\sum_{i \in [n]} p_{i}-p_{n}\left(\sum_{i \in [n-1]} p_{i}\right),1\right).
\end{array}
\end{equation}
\end{theorem}
\begin{proof}
With $p_{ij} = p_ip_j$ and $k = 1$, the optimal value of the primal linear program in \eqref{eq:generalprimal} is bounded since it is feasible and the objective function describes a probability value. The optimality conditions of linear programming states that $\{\theta(\mb{c}); \mb{c} \in \{0,1\}^n\}$ is primal optimal and $\{\lambda_{ij}; (i,j) \in K_n, \lambda_i; i \in [n], \lambda_0\}$ is dual optimal if and only if they satisfy: (i) the primal feasibility conditions in (\ref{eq:generalprimal}), (ii) the dual feasibility conditions in (\ref{eq:generaldual}) and (iii) the complementary slackness conditions given by:
\begin{equation*}
\begin{array}{rlllll}
\displaystyle \left(\sum_{(i,j) \in K_n} \lambda_{ij}{c}_{i}{c}_{j}+ \sum_{i \in [n]} \lambda_{i}{c}_{i} +\lambda_{0}\right)\theta(\mb{c}) & =&  0, & \forall \mb{c} \in \{0,1\}^n: \sum_{t} {c}_{t} = 0,\\
\displaystyle \left(\sum_{(i,j) \in K_n} \lambda_{ij}{c}_{i}{c}_{j}+ \sum_{i \in [n]} \lambda_{i}{c}_{i} +\lambda_{0}-1\right)\theta(\mb{c}) & =&  0, & \forall \mb{c} \in \{0,1\}^n: \sum_{t} {c}_{t} \geq 1.
\end{array}
\end{equation*}
\texitem{(1)} {Proof of tightness of non-trivial bound in \eqref{union1}}: We show that $\overline{P}(n,1,\mb{p}) = \sum_{i \in [n]} p_{i}-p_{n}(\sum_{i \in [n-1]} p_{i})$ which is the non-trivial part of the upper bound in (\ref{union1}) when $\sum_{i \in [n-1]}p_i \leq 1$. Consider the dual feasible solution $\lambda_{0}= 0$, $\lambda_{i}= 1 \; \forall i \in [n]$, $\lambda_{in}= -1 \; \forall i \in [n-1]$ and $\lambda_{ij} = 0$ otherwise.
We verify the tightness of the bound, by showing there exists a primal solution (feasible distribution) which satisfies the complementary slackness conditions. Towards this, observe that from the complementary slackness conditions in (iii) for all values of $\mb{c} \in \{0,1\}^n$ with $\sum_{t \in [n-1]} {c}_{t} \geq 2$ and $c_n = 0$, we have:
\begin{equation*}
\begin{array}{rlllll}
\displaystyle c_n + \sum_{i \in [n-1]} {c}_{i}(1-c_n)-1 > 0  \Longrightarrow   \theta(\mb{c}) = 0.
\end{array}
\end{equation*}
This forces a total of $2^{n-1}-n$ scenarios to have zero probability. Building on this, we set the probabilities of the $2^n$ possible scenarios of $\tilde{\mb{c}}$ as shown in Table \ref{table:hwprobdist}.
The probability of the vector of all zeros (one scenario) is set to $1 - \sum_{i \in [n]}p_i + p_n(\sum_{i \in [n-1]}p_i)$. To match the bivariate probabilities $\mathbb{P}(\tilde{c}_i=1,\tilde{c}_n=0) = p_i(1-p_n)$, we have to then set the probability of the scenario where $c_i = 1, c_n = 0$ and all remaining $c_j = 0$ to $p_i(1-p_n)$. This corresponds to the $n-1$ scenarios in Table \ref{table:hwprobdist}.
\begin{table}[htbp]
\footnotesize
\caption{Probabilities of $2^{n}$ scenarios.}\label{table:hwprobdist}
\begin{center}
\begin{tabular}{llllllll}
\mbox{Scenarios} & $c_1$ & $c_2$ & \ldots & $c_{n-1}$ & $c_n$ &  \mbox{Probability} \\ \hline
 \mbox{$1$ } &0 & 0 & \ldots & 0 & 0 &  $1 - \sum_{i \in [n]}p_i + p_n\left(\sum_{i \in [n-1]}p_i\right)$ \\
\multirow{4}*{$n-1$ $\begin{dcases*} \\ \\ \\ \end{dcases*}$} & 1 & 0 & \ldots & 0 & 0 & $p_1(1-p_n)$ \\
& 0 & 1 & \ldots & 0 & 0 & $p_2(1-p_n)$   \\
& \vdots & \vdots & \vdots & \vdots & \vdots &  \vdots  \\
& 0 & 0 & \ldots &1 & 0 & $p_{n-1}(1-p_n)$  \\
\multirow{2}*{$2^{n-1}-n$ $\begin{dcases*} \\ \\ \\ \end{dcases*}$} & 1 & 1 & \ldots &0 & 0 & $0$  \\
& \vdots & \vdots & \vdots & \vdots & \vdots & \vdots  \\
& 1& 1 & \ldots & 1 & 0 & $0$  \\
  \multirow{2}*{$2^{n-1}$ $\begin{dcases*} \\ \\ \\ \end{dcases*}$}& 0 & 0 & \ldots &0 & 1 & $\theta(\mb{c})$ \multirow{2}*{ $\begin{rcases*} \\ \\ \\ \end{rcases*} p_n$} \\
& \vdots & \vdots &  \vdots & \vdots & \vdots & \vdots  \\
& 1& 1 & \ldots & 1 & 1 &  $\theta(\mb{c})$
\end{tabular}
\end{center}
\end{table}
 Hence, to ensure feasibility of the distribution, we need to show that there exist nonnegative values of $\theta(\mb{c})$ for the last $2^{n-1}$ scenarios such that:
\begin{equation*}
\begin{array}{rlll} \label{compslack2}
\displaystyle \sum_{\mbs{c} \in \{0,1\}^n: c_n = 1} \theta(\mb{c}) & = &p_n, & \\
\displaystyle \sum_{\mbs{c} \in \{0,1\}^n: c_i = 1, c_n = 1} \theta(\mb{c}) & = & p_ip_n, & \forall i \in [n-1],\\
\displaystyle \sum_{\mbs{c} \in \{0,1\}^n: c_i = 1, c_j = 1, c_n = 1} \theta(\mb{c}) & = & p_ip_j, & \forall (i,j) \in K_{n-1},
\end{array}
\end{equation*}
or equivalently, by conditioning on ${c}_n = 1$, we need to show that there exists nonnegative values of $\vartheta(\mb{c}_{-n}) = \mathbb{P}(\tilde{\mb{c}}_{-n} = \mb{c}_{-n})$ for all $\mb{c}_{-n} = (c_1,\ldots,c_{n-1})\in \{0,1\}^{n-1}$ such that:
\begin{equation} \label{eq:compslack3}
\begin{array}{rlll}
\displaystyle \sum_{\mbs{c}_{-n} \in \{0,1\}^{n-1}} \vartheta(\mb{c}_{-n}) & = &1, & \\
\displaystyle \sum_{\mbs{c}_{-n} \in \{0,1\}^{n-1}: c_i = 1} \vartheta(\mb{c}_{-n}) & = & p_i, & \forall i \in [n-1], \\
\displaystyle \sum_{\mbs{c}_{-n} \in \{0,1\}^{n-1}: c_i = 1, c_j = 1} \vartheta(\mb{c}_{-n}) & = & \displaystyle \frac{p_ip_j}{p_n}, & \forall (i,j) \in K_{n-1}.
\end{array}
\end{equation}
This corresponds to verifying the existence of a probability distribution on $n-1$ Bernoulli random variables with univariate probabilities $p_i$ and bivariate probabilities $p_ip_j/p_n$ where $p_{1} \leq p_{2} \leq \ldots \leq p_{n-1} \leq p_n$. Observe that in (\ref{eq:compslack3}), the univariate probabilities remain the same but the random variables are no longer pairwise independent.
Now we make use of Lemma \ref{lem:bivarfeas} to claim that (\ref{eq:compslack3}) is always feasible. By considering $n-1$ variables and setting $p=p_n \geq \max_{i \in [n-1]} p_i $, it is to easy to see from Lemma \ref{lem:bivarfeas} that there exists a distribution which satisfies (\ref{eq:compslack3}). An outline of the different distributions used in the construction is provided in Figure \ref{fig:unionconstruction}.
\begin{figure}[htbp]
 \centering
 \begin{tikzpicture}[thick,scale=0.65, every node/.style={transform shape}]
      \node[ draw,circle, text width=2cm,align=center] (fl1) {\large{$(p_{i},p_{i}p_{j})$}\\~\\  \small{\mbox{n dimensions}}};
      \node[xshift=0 cm,yshift=-4cm,draw, circle, text width=2.2cm,align=center] (fl2) {\large{ $\left(p_i,\frac{p_ip_j}{p_n}\right)$}\\~\\\small{\mbox{n-1 dimensions}}};
        \node[xshift=-3cm,yshift=-7.5cm,draw,circle, text width=2.2cm,align=center] (fl3){\large{$(p_{i},p_ip_j)$}\\~\\\small{\mbox{n-1 dimensions}}};
 \node[xshift=+3cm,yshift=-7.5cm,draw,circle, text width=2.2cm,align=center] (fl4) {\large{$\bigg(p_{i}, \frac{p_ip_j}{p_{n-1}}\bigg)$}\\~\\\small{\mbox{n-1 dimensions}}};
 \node[xshift=+6cm,yshift=-11.1cm,draw,circle, text width=2.2cm,align=center] (fl6)
      {\small{$\bigg( \displaystyle \frac{p_{i}}{p_{n-1}},\frac{p_ip_j}{p_{n-1}^{2}}\bigg)$\\~\\\small{\mbox{n-2 dimensions}}}};
    \path[line] (fl2.north) --node[right]{}(fl1.south);
        \path[line] (fl3.60) --  node[sloped, anchor=center, above]{}(fl2.225);
        \path[line] (fl4.120) --  node[sloped, anchor=center, above]{}(fl2.315);
        \path[line] (fl6.120) -- node[sloped, anchor=center, above]{}(fl4.315);
      \draw (fl1.south) -- (fl2.north);
    \end{tikzpicture}
\caption{Construction of the extremal distribution.}
  \label{fig:unionconstruction}
    \end{figure}
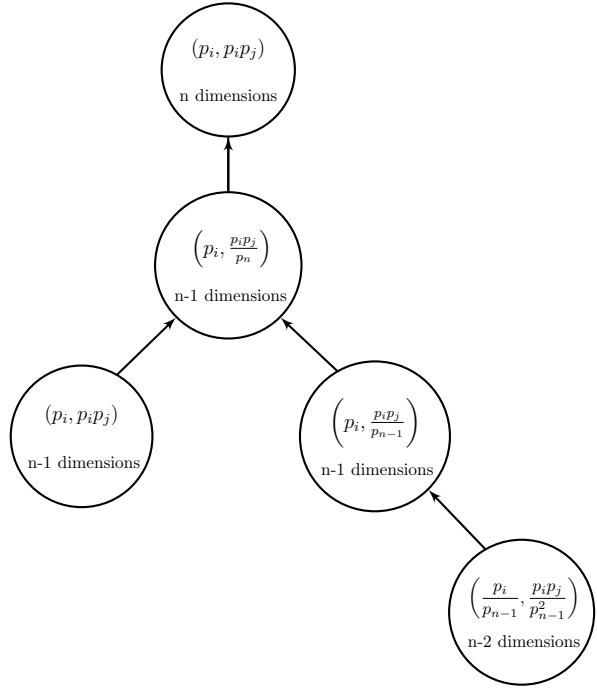
This completes the proof for the case where $\sum_{i \in [n-1]}p_i \leq 1$ with:
\begin{equation*}\label{treeb}
\begin{array}{rlll}
\displaystyle \overline{P}(n,1,\mb{p}) =\sum_{i \in [n]} p_{i}-p_{n}\left(\sum_{i \in [n-1]} p_{i}\right).
\end{array}
\end{equation*}
   \texitem{(2)} {Proof of tightness of the trivial part of the bound in \eqref{union1}}:
To complete the proof, consider the case with $\sum_{i \in [n-1]} p_i > 1$. Then, there exists an index $t \in [2,n-1]$ such that $\sum_{i \in [t-1]} p_i \leq 1$ and $\sum_{i \in [t]} p_i > 1$. Let $\delta = 1 - \sum_{i \in [t-1]} p_i$. Clearly $0 \leq \delta < p_t$. From step (1), we know that there exists a distribution for $t+1$ pairwise independent random variables with marginal probabilities $p_1,p_2,\ldots,p_{t-1},\delta,p_{t+1}$ such that the probability of the sum of the random variables being at least one is equal to one (since the sum of the first $t$ probabilities in this case is equal to one). By increasing the marginal probability $\delta$ to $p_t$, we can only increase this probability. To see this, consider the distribution for $t+1$ mutually independent Bernoulli random variables with marginal probabilities $p_1,p_2,\ldots,p_{t-1},1,p_{t+1}$ where the probability of the sum of the random variables being at least one is equal to one. We can then find a $\lambda \in [0,1)$ such that $p_t = \lambda \delta + (1-\lambda)$ and construct a pairwise independent distribution for $t+1$ pairwise independent random variables with marginal probabilities $p_1,p_2,\ldots,p_{t-1},p_t,p_{t+1}$ by using the convex combination of the two distributions with sum of the random variables taking a value at least one with probability one.
We can generate the remaining random variables $\tilde{c}_{t+2},\ldots,\tilde{c}_n$ independently with marginal probabilities $p_{t+2},\ldots,p_n$. This provides a feasible distribution that attains the bound of one, thus completing the proof.
\end{proof}

\subsection{Connection of Theorem \ref{thm:unionbound} to existing results} \label{subsec:hunter}
Bounds on the probability that the sum of Bernoulli random variables is at least one has been extensively studied in the literature, under knowledge of general bivariate probabilities. Let ${A}_i$ denote the event that ${c}_i = 1$ for each $i$, then, $k = 1$ simply corresponds to bounding the probability of the union of events. When the marginal probabilities $p_i = \mathbb{P}(A_{i})$ for $i \in [n]$ and bivariate probabilities $p_{ij} =\mathbb{P}(A_{i} \cap A_{j})$ for $(i,j) \in K_n$ are given, Hunter \cite{hunter1976} and Worsley \cite{worsley1982} derived the following bound by optimizing over spanning trees $\tau \in T$:
\begin{align}\label{hunter1}
\mathbb{P}(\displaystyle {{\cup}}_i A_{i}) \leq \displaystyle \sum_{i \in [n]} p_{i}-\underset {\tau \in T}{\max}\sum_{(i,j) \in \tau} p_{ij},
\end{align} where $T$ is the set of all spanning trees on the complete graph with $n$ nodes with edge weights given by $p_{ij}$. A special case of the Hunter \cite{hunter1976} bound was derived by Kounias \cite{kounias1968}:
\begin{align}\label{kounias1}
\mathbb{P}(\displaystyle  {{\cup}_i} A_{i}) \leq \displaystyle \sum_{i \in [n]} p_{i}-\underset {j\in [n]}{\max} \sum_{i\neq j} p_{ij},
\end{align}
which subtracts the maximum weight of a star spanning tree from the sum of the marginal probabilities. Tree bounds have been shown to be tight, in some special cases as outlined next:
\texitem{(a)} Zero bivariate probabilities for all pairs: When all the probabilities $p_{ij}$ are zero, the bound reduces to the Boole union bound which is tight.
\texitem{(b)} Zero bivariate probabilities outside a given tree: Given a tree $\tau$ such that the bivariate probabilities $p_{ij}$ are zero for edges $(i,j) \notin \tau$, Worsley \cite{worsley1982} proved that the bound is tight (see Veneziani \cite{veneziani2008hunter} for related results).
\texitem{(c)} Lower bounds on bivariate probabilities: Boros et al. \cite{boros2014} proved that by relaxing the equality of bivariate probabilities to lower bounds on bivariate probabilities: $$\mathbb{P}\big(A_{i} \cap A_{j}\big)\geq p_{ij},\; \forall (i,j) \in K_n ,$$ the tightest upper bound on the probability of the union is exactly the Hunter \cite{hunter1976} and Worsley \cite{worsley1982}  bound (see Maurer \cite{maurer} for related results).
\texitem{(d)} Pairwise independent variables (Theorem \ref{thm:unionbound} in this paper): With pairwise independent random variables where $p_{ij} = p_ip_j$, the maximum weight spanning trees in \eqref{hunter1} is exactly the star tree with the root at node $n$ and edges $(i,n)$ for all $i \in [n-1]$. In, this case, the Kounias \cite{kounias1968}, Hunter \cite{hunter1976} and Worsley \cite{worsley1982} bound reduce to the bound in \eqref{union1} which is shown to be tight in Theorem \ref{thm:unionbound} of this paper.

The next example illustrates that with general bivariate probabilities, even if a joint distribution exists, the Hunter \cite{hunter1976}, Worsley \cite{worsley1982} bound and Kounias \cite{kounias1968} bound are not guaranteed to be tight.

\begin{example} \label{example:hw}
Consider $n=4$ Bernoulli random variables with univariate marginal probabilities: $$p_1=0.35,\;p_2 =0.19,\;p_3 =0.13,\;p_4=0.2,$$ and bivariate probabilities:
$$p_{12}=0.001,\;p_{13}=0.022,\;p_{14}=0.03,\;p_{23}=0.017,\;p_{24}=0.018,\;p_{34}=0.019.$$ It can be verified using linear programming that a joint distribution with these given univariate and bivariate probabilities exists. The tight upper bound obtained by solving the linear program \eqref{eq:generalprimal} is equal to: $$\max_{\theta \in \Theta(\mbs{p},p_{ij};(i,j) \in K_4)} \mathbb{P}_\theta \left(\tilde{c}_{1}+\tilde{c}_{2}+\tilde{c}_{3}+\tilde{c}_{4} \geq 1\right)=0.784.$$
Figure \ref{fig:hwknottight} displays the star spanning tree chosen by the Kounias \cite{kounias1968} bound and the spanning tree chosen by the Hunter \cite{hunter1976} and Worsley \cite{worsley1982} bound. It is clear that none of these bounds are tight in this instance. Boros et al. \cite{boros2014} also provide randomly generated instances (see Table 1 of Section 4 in their paper) where the Hunter \cite{hunter1976} and Worsley \cite{worsley1982} bound is not tight, athough it provides the best performance among the upper bounds considered there.
\begin{figure}[H]
  \centering
   \includegraphics[scale=0.4]{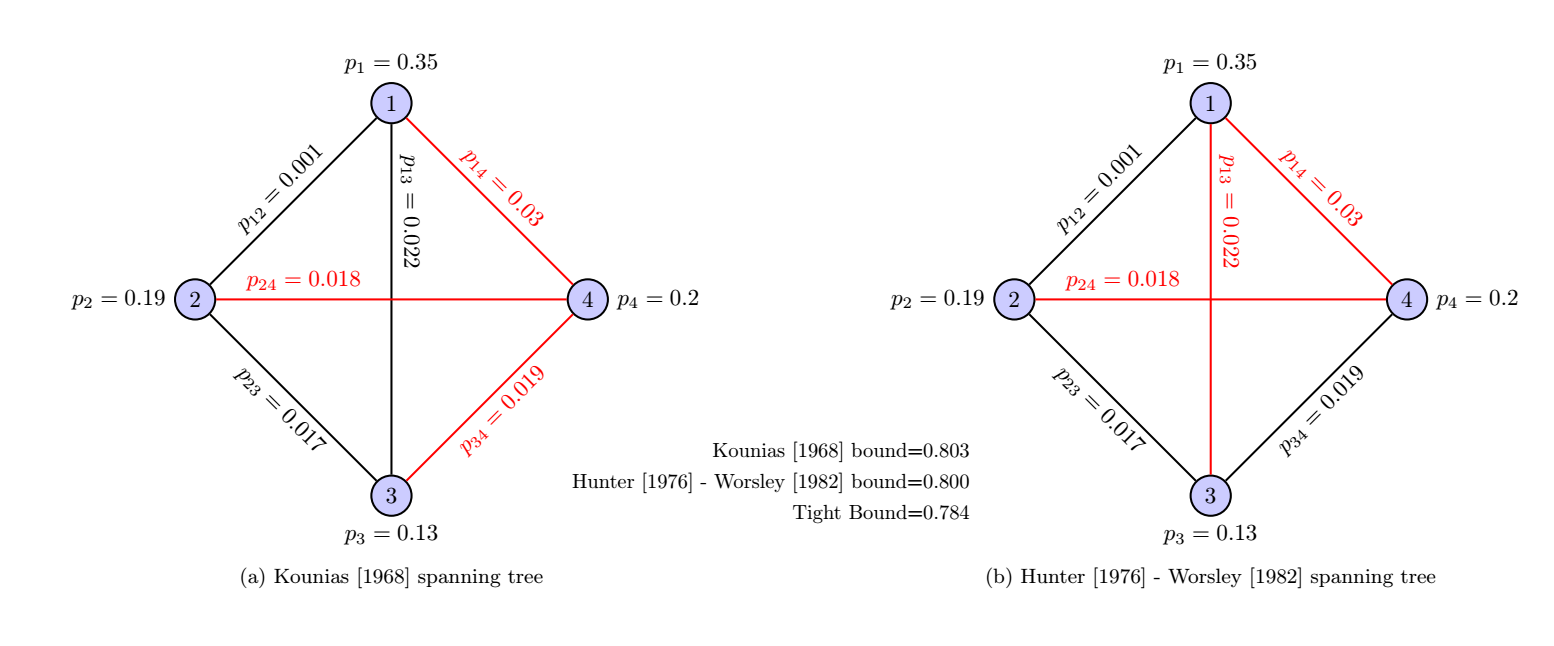}
   \caption{Kounias \cite{kounias1968}, Hunter \cite{hunter1976} and Worsley \cite{worsley1982} spanning trees with general bivariates}
   \label{fig:hwknottight}
 \end{figure}
Figure \ref{fig:all3same} demonstrates that with the same set of univariate marginals, when pairwise independence is enforced, the spanning trees obtained from all these approaches are identical and the bounds in \eqref{hunter1} and \eqref{kounias1} equal the tight bound $0.688$ (from Theorem \ref{thm:unionbound}).
\begin{figure}[H]
  \centering
\includegraphics[scale=0.4]{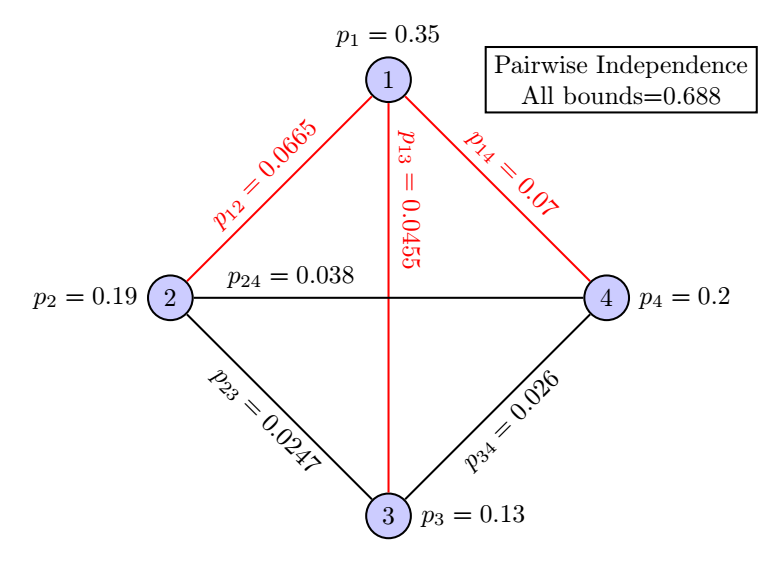}
\caption{Optimal spanning tree with pairwise independence when $\mb{p}=(0.35,0.19,0.13,0.2)$.}
\label{fig:all3same}
 \end{figure}
\end{example}
\subsection{Comparison with the union bound} \label{subsec:43bound}

The next proposition provides an upper bound on the ratio of the Boole union bound and the pairwise independent bound in (\ref{union1}) in Theorem \ref{thm:unionbound}.
\begin{prop} \label{prop:25}
For all $\mb{p} \in [0,1]^n$, we have:
$$\displaystyle \frac{{\overline{P}}_u(n,1,\mb{p})}{{\overline{P}}(n,1,\mb{p})} \leq \frac{4}{3}.$$
The ratio of $4/3$ is attained when $\sum_{i \in [n-1]} p_{i}=1/2$ and $p_n=1/2$.
\end{prop}
\begin{proof}
Assume the probabilities are sorted in increasing value as $0 \leq p_{1} \leq p_{2} \leq \ldots \leq p_{n} \leq 1$. It is straightforward to see that if $\sum_{i \in [n-1]} p_i > 1$, both the bounds take the value of $\overline{P}(n,1,\mb{p}) = {\overline{P}}_u(n,1,\mb{p}) = 1$. Now assume, $\alpha = \sum_{i \in [n-1]} p_{i} \leq 1$. The ratio is given as:
\begin{equation*}
\begin{array}{llll}
\displaystyle \frac{\overline{P}_u(n,1,\mb{p})}{{\overline{P}}(n,1,\mb{p})} & =& \displaystyle \frac{\min\left(\sum_{i \in [n]} p_{i},1\right)}{ \sum_{i \in [n]} p_{i}-p_{n}\left(\sum_{i \in [n-1]} p_{i}\right)}\\
& = & \displaystyle \frac{\min\left(\alpha+p_n,1\right)}{\alpha+p_n-\alpha p_{n}}.
\end{array}
\end{equation*}
If $\alpha+p_n \leq 1$, then we have:
\begin{equation*}
\begin{array}{llll} \label{bestp_n}
\displaystyle \frac{\overline{P}_u(n,1,\mb{p})}{{\overline{P}}(n,1,\mb{p})}&=&\displaystyle \frac{\alpha+p_n}{\alpha+p_n-\alpha p_n}\\
&  = &\displaystyle \frac{1}{1-\frac{1}{\frac{1}{\alpha}+\frac{1}{p_n}}}\\
& \leq & \displaystyle \frac{4}{3}\\
& & [\mbox{where the maximum is attained at } \alpha = 1-p_n \mbox{ and } p_n = 1/2].
\end{array}
\end{equation*}
If $\alpha+p_n \geq 1 $, then we have:
\begin{equation*}
\begin{array}{llll} \label{bestp_n}
\displaystyle \frac{\overline{P}_u(n,1,\mb{p})}{{\overline{P}}(n,1,\mb{p})}&=&\displaystyle \frac{1}{\alpha+p_n-\alpha p_{n}}\\
&=&\displaystyle \frac{1}{\alpha(1-p_n)+p_{n}}\\
& \leq & \displaystyle \frac{4}{3}\\
& & [\mbox{where the maximum is attained at } \alpha = 1-p_n \mbox{ and } p_n = 1/2].
\end{array}
\end{equation*}
This gives the bound of $4/3$ when $p_n = 1/2$ and $\alpha = 1/2$.
\end{proof}
We next illustrate an application of Theorem \ref{thm:unionbound} and Proposition \ref{prop:25} in comparing bounds with dependent and independent random variables in correlation gap analysis.
\begin{example} [Correlation gap analysis] \label{ex:correlationanalysis}
The notion of ``correlation gap'' was introduced by Agrawal et al. \cite{Agrawal2012}. It is defined as the ratio of the worst-case expected cost for random variables with given univariate marginals to the expected cost when the random variables are independent. When $\tilde{\mb{c}}$ is a Bernoulli random vector and $\theta_{ind}$ denotes the independent distribution, the correlation gap is defined as:
\begin{equation}
\begin{array}{rlll}\label{correlation1}
\kappa_u(\mb{p}) & = &  \underset{\theta \in \Theta(\mbs{p})}{\sup} \quad \dfrac{\mathbb{E}_{\theta}[f(\tilde{\mb{c}})]}{\mathbb{E}_{\scalebox{0.65}{${\theta}_{ind}$}}[f(\tilde{\mb{c}})]}.
\end{array}
\end{equation}
A function $f: \{0,1\}^n \rightarrow \mathbb{R}_+$ is: (i) submodular if $f(\mb{c}) + f(\mb{d}) \geq  f(\mb{c} \wedge \mb{d})+f(\mb{c} \vee \mb{d})$ for all $\mb{c}, \mb{d} \in \{0,1\}^n$ with $\mb{c} \wedge \mb{d} = (\min(c_1,d_1),\ldots,\min(c_n,d_n))$ and $\mb{c} \vee \mb{d} = (\max(c_1,d_1),\ldots,\max(c_n,d_n))$ and (ii) nondecreasing if $f(\mb{c}) \geq f(\mb{d})$ for all $\mb{c} \geq \mb{d}$. A key result in this area is that for any nonnegative, nondecreasing, submodular function, the correlation gap is always upper bounded by $e/(e-1)$ (see Calinescu et al. \cite{Calinescu2}, Agrawal et al. \cite{Agrawal2012}). The example constructed in these papers show the bound is attained for the maximum of binary variables $ f(\mb{c}) = \max_{i \in [n]} c_i$.
For a given marginal vector $\mb{p}$, the correlation gap in \eqref{correlation1} reduces to:
\begin{equation}\label{correlation2}
\begin{array}{rlll}
\kappa_u(\mb{p})& = & \displaystyle \frac{\max_{\theta \in \Theta(\mbs{p})}\mathbb{E}_{\theta}[\max\left(\tilde{c}_1,\tilde{c}_2,...,\tilde{c}_n\right)]}{1-\prod_{i=1}^n (1-p_i)}\\
& = &\displaystyle \frac{\max_{\theta \in \Theta(\mbs{p})}\mathbb{P}_\theta \left(\sum_{i \in [n]} \tilde{c}_i \geq 1\right)}{1-\prod_{i=1}^n (1-p_i)}\\
& = &\displaystyle \frac{\min\left(\sum_{i \in [n]} p_i,\;1\right)}{1-\prod_{i=1}^n (1-p_i)}.
\end{array}
\end{equation}
We now provide an extension of this definition by considering the ratio of the worst-case expected cost when the random variables are pairwise independent to the expected cost when the random variables are independent. This is given as:
 \begin{equation*}
\begin{array}{rlll}
\kappa(\mb{p}) & = &  \underset{\theta \in \Theta(\mbs{p},p_{ij}; (i,j) \in K_n)}{\sup} \quad \dfrac{\mathbb{E}_{\theta}[f(\mb{\tilde{c}})]}{\mathbb{E}_{\scalebox{0.65}{${\theta}_{ind}$}}[f(\mb{\tilde{c}})]},
\end{array}
\end{equation*}
which reduces in this specific case to:
 \begin{equation*}
\begin{array}{rlll}
\kappa(\mb{p}) & = &  \displaystyle \frac{\min\left(\sum_{i \in [n]} p_{i}-p_{n}\left(\sum_{i \in [n-1]} p_{i}\right),1\right)}{1-\prod_{i=1}^n (1-p_i)}.
\end{array}
\end{equation*}
Clearly $\kappa(\mb{p}) \leq \kappa_u(\mb{p})$. We next compare these two ratios.
\texitem{(a)} Worst-case analysis:
Assume the marginal probability vector is given by $\mb{p}=(1/n,\ldots,1/n)$. For the independent distribution, the probability is given by $1-(1-1/n)^n$, while the Boole union bound is equal to one (attained by the distribution which assigns probability $1/n$ to each of $n$ support points with $c_i = 1,\;c_j=0, \forall j \neq i$ (for each $i \in [n])$ and zero otherwise). In this case, the limit of the ratio as $n$ goes to infinity is given by:$$\displaystyle \lim_{n \to \infty}\kappa_u(\mb{p})  = \frac{1}{1-(1-1/n)^n} = \frac{e}{e-1} \approx 1.5819.$$ Likewise it is easy to verify that with pairwise independence:
 $$\displaystyle \lim_{n \to \infty}\kappa(\mb{p})  =  \frac{1-{1}/{n}\left(1-{1}/{n}\right)}{1-(1-1/n)^n} = \frac{e}{e-1} \approx 1.5819.$$
Thus in the worst-case, both these bounds attain the ratio $e/(e-1)$.
\texitem{(b)} Instances where the correlation gap can be improved:
On the other hand, Proposition \ref{prop:25} illustrates that for the probabilities $p_n = 1/2$ and $\sum_{i \in [n-1]} p_i = 1/2$, the pairwise independent bound is $3/4$ and the Boole union bound is one. For example with $n = 2$ where $\mb{p} = (1/2,1/2)$, the Boole union bound is one, while both the pairwise independent bound and the independent probability is equal to $3/4$. Then, we have $\kappa_u((1/2,1/2)) = 4/3$ while $\kappa((1/2,1/2)) = 1$. Thus in specific instances, the correlation gap can be tightened by considering pairwise independent random variables.
\end{example}
An application of the $4/3$ bound in Proposition \ref{prop:25} in the context of distributionally robust optimization is discussed next.
\begin{example} [Distributionally robust bottleneck combinatorial optimization] \label{ex:bottleneck}
Consider a set of $n$ elements indexed by $[n] = \{1,2,\ldots,n\}$ where element $i$ has a cost of $c_i$. Given a set of feasible solutions ${\cal X} \subseteq \{0,1\}^n$, the goal in the bottleneck combinatorial optimization problem is to find the solution $\mb{x} \in {\cal X}$ that minimizes the maximum cost among the selected elements (bottleneck cost). This is formulated as the bottleneck combinatorial optimization problem:
\begin{equation*}
\begin{array}{rlllll}
\displaystyle \min_{\mbs{x} \in {\cal X} \subseteq \{0,1\}^n} \max_{i \in [n]} c_ix_i.
\end{array}
\end{equation*}
A threshold algorithm to solve this class of problems was developed by Edmonds and Fulkerson \cite{Edmonds}. Consider a distributionally robust variant of this problem where the cost of the element $i$ is a random variable $\tilde{c}_i$ and the joint distribution of $\tilde{\mb{c}}$ is not fully specified.  The distributionally robust bottleneck optimization problem is formulated as:
\begin{equation*}
\begin{array}{rlllll}
\displaystyle \min_{\mbs{x} \in {\cal X} \subseteq \{0,1\}^n} {\max_{\theta \in \Theta}\mathbb{E}\left[\max_{i \in [n]}\tilde{c}_ix_i\right]},
\end{array}
\end{equation*}
where $\Theta$ is the set of possible joint distributions and the goal is to find the solution $\mb{x} \in {\cal X}$ that minimizes the maximum expected bottleneck cost. Such problems have been studied in Agrawal et al. \cite{Agrawal2012} where the distributions are specified up to marginal information and Xie et al. \cite{Xie} where the distributions are assumed to lie in a ball around an empirical distribution specified by the Wasserstein distance. Here we consider the set of distributions with pairwise independent random variables where $\Theta = \Theta(\mb{p},p_ip_j; (i,j) \in K_n)$. The next proposition provides a $4/3$-approximation algorithm for this problem.
\begin{prop} \label{prop:bottleneck}
Let $\mbox{OPT}$ be the optimal value of the distributionally robust bottleneck combinatorial optimization problem:
$$\displaystyle \mbox{OPT} = \min_{\mbs{x} \in {\cal X} \subseteq \{0,1\}^n} \underbrace{\max_{\theta \in \Theta(\mbs{p},p_ip_j; (i,j) \in K_n)}\mathbb{E}\left[\max_{i \in [n]}\tilde{c}_ix_i\right]}_{f(\mbs{x})}.$$
Suppose we can optimize linear functions over the set ${\cal X} \subseteq \{0,1\}^n$ in polynomial time. Then, we can find $\hat{\mb{x}}$ in polynomial time such that:  $$\displaystyle \mbox{OPT} \leq f(\hat{\mb{x}}) \leq \frac{4}{3}\mbox{OPT}.$$
\end{prop}
\begin{proof}
When $\mb{x} \in {\cal X} \subseteq \{0,1\}^n$, each $\tilde{c}_ix_i$ is a Bernoulli random variable with $\mathbb{P}(\tilde{c}_{i}x_i = 1) = p_ix_i$. Using the Boole union bound, we have:
\begin{equation*}
\begin{array}{rlllll}
\displaystyle \max_{\theta \in \Theta(\mbs{p})}\mathbb{E}\left[\max_{i \in [n]}\tilde{c}_i{x}_i\right] = \min\left(1,\sum_{i \in [n]} p_ix_i\right).
\end{array}
\end{equation*}
Consider the solution $\hat{\mb{x}}$ which is computable in polynomial time by solving the minimum cost combinatorial optimization problem:
\begin{equation*}
\begin{array}{rlllll}
\displaystyle \hat{\mb{x}} \in \arg\min_{\mbs{x} \in {\cal X} \subseteq \{0,1\}^n} \sum_{i \in [n]} p_ix_i.
\end{array}
\end{equation*}
Let $\mb{x}^*$ denote the optimal solution and $\theta^*$ denote the worst-case pairwise independent distribution in OPT. Then we have:
\begin{equation*}
\begin{array}{rlll}
\displaystyle \frac{f(\hat{\mb{x}})}{\mbox{OPT}}& = & \displaystyle  \frac{\max_{\theta \in \Theta(\mbs{p},p_ip_j; (i,j) \in K_n)}\mathbb{E}\left[\max_{i \in [n]}\tilde{c}_i\hat{x}_i\right]}{ {\mathbb{E}_{\theta^*}\left[\max_{i \in [n]}\tilde{c}_ix_i^*\right]}}\\
& \leq & \displaystyle \frac{\max_{\theta \in \Theta(\mbs{p})}\mathbb{E}\left[\max_{i \in [n]}\tilde{c}_i\hat{x}_i\right]}{{\mathbb{E}_{\theta^*}\left[\max_{i \in [n]}\tilde{c}_ix_i^*\right]}}\\
& & [\mbox{since } \Theta(\mb{p},p_ip_j; (i,j) \in K_n) \subseteq \Theta(\mb{p})]\\
& = &\displaystyle \frac{\min\left(1,\sum_{i \in [n]}p_i\hat{x}_i\right)}{ {\mathbb{E}_{{\theta}^*}\left[\max_{i \in [n]}\tilde{c}_ix_i^*\right]}}\\
& \leq &\displaystyle \frac{\min\left(1,\sum_{i \in [n]}p_i{x}_i^*\right)}{ {\mathbb{E}_{{\theta}^*}\left[\max_{i \in [n]}\tilde{c}_ix_i^*\right]}}\\
&  & [\mbox{since } \mb{x}^* \mbox{ is only feasible for the sum objective}]\\
& = & \displaystyle \frac{{\overline{P}}_u(n,1,\mb{p}\cdot{\mb{x}}^*)}{{\overline{P}}(n,1,\mb{p}\cdot\mb{x}^*)}\\
&  & [\mbox{where } \mb{p}\cdot{\mb{x}}^* = (p_1x_1^*,\ldots,p_nx_n^*)]\\
& \leq & \frac{4}{3}\\
& & [\mbox{from Proposition \ref{prop:25}}].
\end{array}
\end{equation*}
\end{proof}
\end{example}
Proposition \ref{prop:bottleneck} can be applied to instances such as the bottleneck assignment, bottleneck matching problem and bottleneck shortest path problems and provides a $4/3$-approximation for these instances.
The next result shows that Theorem \ref{thm:unionbound} can be used to prove a tight lower bound on the probability of the intersection of pairwise independent events.
\subsection{Tight lower bound for $k = n$} \label{sec:intersectionbound}
\noindent Denote the tightest lower bound on the probability of the intersection of pairwise independent events by $\underline{P}(n,n,\mb{p})$.
Then,
\begin{equation*}
\begin{array}{lll}
\displaystyle \underline{P}(n,n,\mb{p}) =  \displaystyle \min_{\theta \in \Theta(\mbs{p},p_ip_j; (i,j) \in K_n)} \mathbb{P}_\theta\left( \sum_{i \in [n]} \tilde{c}_{i} =n\right).
\end{array}
\end{equation*}
\begin{corollary}\label{cor:intersection}
Sort the probabilities in increasing value as $0 \leq p_{1} \leq p_{2} \leq \ldots \leq p_{n} \leq 1$. Then,
\begin{equation}
 \label{eq:intersectionbound}
 \begin{array}{rlllll}
\displaystyle \underline{P}(n,n,\mb{p})  =  \displaystyle \max\left(p_{1}\left(\sum_{i=2}^{n} p_{i}-(n-2)\right),0\right).
\end{array}
\end{equation}
\end{corollary}
\begin{proof}
The proof follows from that of the union probability bound in Theorem \ref{thm:unionbound}.
Define a complementary Bernoulli random variable $d_{i}=1-c_{n-i+1}, \;i \in [n]$, with transformed probabilities $\mathbb{P}(\tilde{d}_i = 1) = q_i=1-p_{n-i+1}, \; i \in [n]$ and thus $0 \leq q_{1} \leq q_{2} \leq \ldots \leq q_{n} \leq 1$.  We first note that the maximum probability of the union of pairwise independent events can be expressed as an equivalent maximization problem defined on $\mb{d}$ as follows:
\begin{equation}\label{eq:dual<=n-1}
\begin{array}{llllll}
\displaystyle \overline{P}(n,1,\mb{p}) =\displaystyle \overline{Q}(n,n-1,\mb{q}) = \displaystyle \max_{\theta \in \Theta(\mbs{q},q_iq_j; (i,j) \in K_n)} \mathbb{P}_\theta\left( \sum_{i \in [n]} \tilde{d}_{i} \leq n-1\right)
  \end{array}
\end{equation}
\noindent where $\overline{Q}(n,n-1,\mb{q})$ is the maximum probability that at most $n-1$ complimentary events occur.
\noindent The proof is then completed by noting that the tight lower intersection bound $\underline{P}(n,n,\mb{q})$ can be expressed as
\begin{equation*}
 \begin{array}{rlllll}
\displaystyle \underline{P}(n,n,\mb{q})  & =& 1-\overline{Q}(n,n-1,\mb{q}) \\
& =& 1-\overline{P}(n,1,\mb{p}) \\
& = & 1-
\min\left(\sum_{i \in [n]} p_{i}-p_{n}\left(\sum_{i \in [n-1]} p_{i}\right),1\right)\\
 & = & 1-\min\left(1-\left(1-p_{n}\right)\left(1-\sum_{i \in [n-1]} p_{i}\right),1\right)\\
& = &   \displaystyle \max\left(q_{1}\left(\sum_{i=2}^{n} q_{i}-(n-2)\right),0\right).
\end{array}
\end{equation*}
and replacing $\mb{q}$ by $\mb{p}$.
\end{proof}
\medskip
\noindent \textbf{Extremal Distribution}:
The primal distribution which attains the non-trivial part of the tight intersection bound $\underline{P}(n,n,\mb{q})$ is shown in Table \ref{table:comphwprobdist}. It can be constructed from the union probability extremal distribution $\theta^{\star}$ in Table \ref{table:hwprobdist} by flipping the zeros and one's of the support, reversing the bits (to ensure ordering of the transfomed probabilities) and retaining the same joint probabilities $\theta^{\star} (\mb{c})$ but expressed in terms of $\mb{q}$ instead of $\mb{p}$.
\begin{table}[H]
\footnotesize
\caption{Probabilities of $2^{n}$ scenarios.}\label{table:comphwprobdist}
\begin{center}
\begin{tabular}{llllllll}
\mbox{Scenarios} & $d_1$ & $d_2$ & \ldots & $d_{n-1}$ & $d_n$ &  \mbox{Probability} \\ \hline
  \multirow{2}*{$2^{n-1}$ $\begin{dcases*} \\ \\ \\ \end{dcases*}$}& 0 & 0 & \ldots &0 & 0 & $\theta(\mb{d})$\; \multirow{2}*{ $\begin{rcases*} \\ \\ \\ \end{rcases*} 1-q_1$} \\
& \vdots & \vdots &  \vdots & \vdots & \vdots & \vdots  \\
& 0& 1 & \ldots & 1 & 1 &  $\theta(\mb{d})$ \\
\multirow{2}*{$2^{n-1}-n$ $\begin{dcases*} \\ \\ \\ \end{dcases*}$} & 1 & 0 & \ldots &0 & 0 & $0$  \\
& \vdots & \vdots & \vdots & \vdots & \vdots & \vdots  \\
& 1& 1 & \ldots & 1 & 0 & $0$  \\
\multirow{4}*{$n-1$ $\begin{dcases*} \\ \\ \\ \end{dcases*}$} & 1 & 0 & \ldots & 1 & 1& $q_1(1-q_2)$ \\
& \vdots & \vdots & \vdots & \vdots & \vdots &  \vdots  \\
& 1 & 1 & \ldots & 0 & 1 & $q_1(1-q_{n-1})$   \\
& 1 & 1 & \ldots &1 & 0 & $q_{1}(1-q_n)$  \\
 \mbox{$1$ } &1 & 1 & \ldots & 1 & 1 &  $q_{1}\left(\sum_{i=2}^{n} q_{i}-(n-2)\right)$ \\
\end{tabular}
\end{center}
\end{table}

\noindent Note that the feasibility of the joint distribution in Table \ref{table:comphwprobdist} depends on the existence of nonnegative values $\theta(\mb{d})$ for the first $2^{n-1}$ scenarios or alternatively by conditioning on $d_1=0$, there exist nonnegative values of $\vartheta(\mb{d}_{-1}) = \mathbb{P}(\tilde{\mb{d}}_{-1} = \mb{d}_{-1})$ for all $\mb{d}_{-1} = (d_2\ldots,d_{n}) \in \{0,1\}^{n-1}$ such that:
\begin{equation}\label{eq:alterbivar}
\scalebox{0.95}{$
\begin{array}{rlll}
\displaystyle \sum_{\mbs{d}_{-1} \in \{0,1\}^{n-1}} \vartheta(\mb{d}_{-1}) & = &1, & \\
\displaystyle \sum_{\mbs{d}_{-1} \in \{0,1\}^{n-1}: d_i = 0} \vartheta(\mb{d}_{-1}) & = & \displaystyle 1-q_{i}, & \forall i \in [2,n], \\
\displaystyle \sum_{\mbs{d}_{-1} \in \{0,1\}^{n-1}: d_i = 0, d_j = 0} \vartheta(\mb{d}_{-1}) & = & \displaystyle \frac{(1-q_i)(1-q_j)}{1-q_{1}}, & \forall (i,j) \in \{(i,j): 2 \leq i < j \leq n\},
\end{array}$}
\end{equation}
\noindent where the constraints in \eqref{eq:alterbivar} is expressed in terms of non-occurence of the Bernoulli events represented by ${\mb{d}}$, \emph{i.e.} $d_i=0$ instead of $d_i=1$.
The existence of such a feasible bivariate distribution $\vartheta$  can be independently verified from Corollary \ref{cor:bivarfeasvariant} by noting that the Bernoulli random vector defined there satisfies
 $\mathbb{P}\left(\tilde{\mb{c}_i}=0\right)=1-p_i, \;  \forall i \in [n]$ and $\mathbb{P}\left(\tilde{\mb{c}_i}=0,\tilde{\mb{c}_j}=0\right)=(1-p_i)(1-p_j)/(1-p)$ for all $(i,j) \in K_{n}$,
subsequently replacing $p_i$ by $q_{i}$ and setting $p=q_1<=\min_{i \in [2,n]} q_i$ for $n-1$ variables instead of $n$. 

\subsubsection{Connection of Corollary \ref{cor:intersection} to existing results}
The intersection bound $\underline{P}(n,n,\mb{p})$ derived in Corollary \ref{cor:intersection} is zero when $\sum_{i=2}^n p_i \leq n-2$.  In related work with identical probabilities $p$, Benjamini et al. \cite{benjamini12} compute that the minimum intersection probability for $t$-wise independent Bernoulli random variables and identify when it is zero. They prove that $\underline{P}(n,n,p)=0$ for all $t <n$ and $p\leq 1/2$ which matches our result with pairwise independence ($t=2$) since $p \leq (n-2)/(n-1) \leq 1/2$ for all  $n\geq 3$.
We will show in Section \ref{subsec:idenBP} that with pairwise independent identical Bernoulli's, it is possible to derive closed-form tight upper and lower bounds on the intersection probability and more generally $\overline{P}(n,k,\mb{p})$ and $\underline{P}(n,k,\mb{p})$ for any $k \in [n]$.
 With arbitrary dependence among the Bernoulli random variables, the Fr\'{e}chet \cite{frechet1935} lower intersection bound is given as:
\begin{equation}\label{frechet}
\begin{array}{rlllll}
{\underline{P}}_u(n,n,\mb{p})  =   \min_{\theta \in \Theta(\mbs{p})} \mathbb{P}_\theta\left( \sum_{i \in [n]} \tilde{c}_{i} = n\right)   = \max\left(\sum_{i \in [n]} p_{i}-(n-1),0\right).
\end{array}
\end{equation}
Clearly, $\underline{P}(n,n,\mb{p}) \geq {\underline{P}}_u(n,1,\mb{p})$ and the lower bound is thus improved with pairwise independence.

\section{Improved bounds with non-identical marginals for $k\geq 2$} \label{sec:nonidenticalnewbounds}
In the previous section, we resolved the question of finding the tightest bound on the probability of the union of pairwise independent events. We now shift attention to the case of at least $k$ pairwise independent events occurring where $k \geq 2$. Deriving tight bounds for general $k$ appears to be challenging. We exploit the ordering of the probabilities to provide new upper bounds by creating feasible solutions to the dual linear program in \eqref{eq:generaldual}. We make use of the observation that all three bounds in \eqref{cheby}, \eqref{SSS} and \eqref{BorosPrekopa} can be expressed in terms of the first two aggregated (or equivalently binomial) moments of the sum of pairwise independent random variables with $S_{1} = \sum_i p_{i}$ and $S_{2}= \sum_{(i,j) \in K_n} p_{i}p_{j}$. The new ordered bounds improve on these three closed-form bounds. We will refer to the original bounds in \eqref{cheby}, \eqref{SSS} and \eqref{BorosPrekopa} as unordered bounds from this point onwards. The next theorem provides probability bounds for the sum of pairwise independent random variables with possibly non-identical marginals when $k \geq 2$.

\begin{theorem} \label{thm:nonidentical}
Sort the input probabilities in increasing order as $p_{1} \leq \ldots \leq p_{n}$. Define the partial binomial moment $S_{1r} = \sum_{i \in [n-r]}p_{i}$ for $r \in [0,n-1]$ and $S_{2r} =  \sum_{(i,j) \in K_{n-r}}p_{i}p_{j}$ for $r \in [0,n-2]$.
\texitem{(a)} The ordered Schmidt, Siegel and Srinivasan bound is a valid upper bound on $\overline{P}(n,k,\mb{p})$:
\begin{equation}
\begin{array}{rllll}   \label{POSB1}
\overline{P}(n,k,\mb{p})
\leq  \displaystyle
   \min\left(1,\min_{r_1 \in [0,k-1]} \left(\frac{S_{1r_1}}{k-r_1}\right), \min_{r_2 \in [0,k-2]}\left(\frac{S_{2r_2}}{\binom{k-r_2} {2}}\right)  \right),  \forall k \in [2,n].
\end{array}
\end{equation}
\texitem{(b)} The ordered Boros and Pr\'{e}kopa bound is a valid upper bound on $\overline{P}(n,k,\mb{p})$:
\begin{equation}\label{Prekopa3}
\begin{array}{llll}
 \overline{P}(n,k,\mb{p})  \leq  \displaystyle \underset{r \in [0,k-1]} {\min}  BP(n-r,k-r,\mb{p}),& \forall k \in [2,n],
\end{array}
\end{equation}
where:
\begin{equation*}
\begin{array}{lll}   \label{Prekopa2}
BP(n-r,k-r,\mb{p}) \\
=\begin{cases}
 1,   &  \displaystyle k< \frac{(n-r-1)S_{1r}-2S_{2r}}{n-r-S_{1r}} +r, \\
 &\\
\displaystyle \frac{(k+n-2r-1)S_{1r}-2S_{2r}}{(k-r)(n-r)} ,& \displaystyle \frac{(n-r-1)S_{1r}-2S_{2r}}{n-r-S_{1r}}+r \leq k <1+\frac{2S_{2r}}{S_{1r}}+r,\\
 &\\
\displaystyle \frac{(i-1)(i-2S_{1r})+2S_{2r}}{(k-r-i)^2+(k-r-i)} ,&\displaystyle k \geq 1+\frac{2S_{2r}}{S_{1r}}+r.\\
\end{cases}
\end{array}
\end{equation*}
and $i=\lceil{((k-r-1)S_{1r}-2S_{2r})}/({k-r-S_{1r}})\rceil$.
\texitem{(c)} The ordered Chebyshev bound is a valid upper bound on $\overline{P}(n,k,\mb{p})$:
\begin{equation}\label{chebyffinal}
\begin{array}{llll}
 \overline{P}(n,k,\mb{p})  \leq  \displaystyle \underset{r \in [0,k-1]} {\min}  CH(n-r,k-r,\mb{p}), \forall k \in [2,n],
\end{array}
\end{equation}
where:
\begin{equation*}
\begin{array}{lll}
CH(n-r,\; k-r,\;\mb{p}) = \begin{cases}
 1,   & k < S_{1r}+r,\\
 \displaystyle \frac{S_{1r}-(S_{1r}^2-2S_{2r})}{S_{1r}-(S_{1r}^2-2S_{2r}) +( k-r-S_{1r})^2} ,& S_{1r} +r \leq k \leq n.
  \end{cases}\\
\end{array}
\end{equation*}
\end{theorem}

\begin{proof}\
 \texitem{(a)} We observe that for any $r_1 \in [0,k-1]$ and any subset $S \subseteq [n]$ of the random variables of cardinality $n-r_1$, an upper bound is given by:
 \begin{equation*}  \label{OS2 derivationa}
\begin{array}{llll}
\displaystyle \mathbb{P}\left(\sum_{i \in [n]} \tilde{c}_i \geq k\right)  & \leq & \displaystyle \mathbb{P}\left(\sum_{i \in S} \tilde{c}_{i} \geq k-r_1\right) \\
&& [\mbox{since } \sum_{i \in [n]} {c}_i \geq k \mbox{ implies } \sum_{i \in S} {c}_{i} \geq k-r_1]\\
&\leq &\displaystyle \frac{\mathbb{E} \left[ \sum_{i \in S}\tilde{c}_{i} \right]}{k-r_1} \\
& &[\mbox{using Markov inequality}]\\
&=&\displaystyle \frac{\sum_{i \in S} p_i}{k-r_1}. &
\end{array}
\end{equation*}
The tightest upper bound of this form is obtained by minimizing over all $r_1 \in [0,k-1]$ and subsets $S \subseteq [n]$ with $|S| = n-r_1$:
\begin{equation}  \label{OS2 derivationaaaa}
\begin{array}{llll}
\displaystyle \mathbb{P}\left(\sum_{i \in [n]} \tilde{c}_i \geq k\right)  & \leq & \displaystyle \min_{r_1 \in [0,k-1]}\;\min_{S:|S| = n-r_1}\frac{\sum_{i \in S} p_i}{k-r_1}  & \\
 &=&\displaystyle \min_{r_1 \in [0,k-1]}\frac{\sum_{i \in [n-r_1]} p_{i}}{k-r_1}\\
 & &[\mbox{using the } n-r_1 \mbox{ smallest probabilities}].
\end{array}
\end{equation}
We derive the other term  in (\ref{POSB1}) using a similar approach while accounting for pairwise independence. For any $r_2 \in [0,k-2]$ and any subset $S \subseteq [n]$ of the random variables of cardinality $n-r_2$, an upper bound is given by:
\begin{equation*}  \label{OS2 derivation}
\begin{array}{llll}
\displaystyle \mathbb{P}\left(\sum_{i \in [n]} \tilde{c}_i \geq k\right)  & \leq & \displaystyle \mathbb{P}\left(\sum_{i \in S} \tilde{c}_{i} \geq k-r_2\right)  &\\
 &=&\displaystyle \mathbb{P}\left(\binom {\sum_{i\in S} \tilde{c}_{i}}{2} \geq  \binom{k-r_2}{2}\right)& \\
&\leq &\displaystyle \frac{\mathbb{E} \left[ \sum_{i \in S} \sum_{j \in S: j > i}\tilde{c}_{i}\tilde{c}_{j}  \right]}{\binom{k-r_2}{2}}& \\
& &[\mbox{using equation } \eqref{binomial} \mbox{ and Markov inequality}]\\
&=&\displaystyle \frac{\sum_{i \in S} \sum_{j \in S: j > i} \mathbb{E}[\tilde{c}_{i}]\mathbb{E}[\tilde{c}_{j}]}{\binom{k-r_2}{2}} & \\
& & [\mbox{using pairwise independence}]\\
&=&\displaystyle \frac{\sum_{i \in S} \sum_{j \in S: j > i}p_i p_j}{\binom{k-r_2}{2}}.&
\end{array}
\end{equation*}
The tightest upper bound of this form is obtained by minimizing over $r_2 \in [0,k-2]$ and all sets $S$ of size $n-r_2$. This gives:
\begin{equation}  \label{OS2 derivationa}
\begin{array}{llll}
\displaystyle \mathbb{P}\left(\sum_{i \in [n]} \tilde{c}_i \geq k\right)  & \leq & \displaystyle \min_{r_2 \in [0,k-2]}\;\min_{S:|S| = n-r_2}\frac{\sum_{i \in S} \sum_{j \in S: j > i}p_i p_j}{\binom{k-r_2}{2}}&\\
 &=&\displaystyle \min_{r_2 \in [0,k-2]}\left(\frac{\sum_{(i,j) \in K_{n-r_2}} p_{i}p_{j}}{\binom{k-r_2} {2}}\right)\\
  & &[\mbox{using the } n-r_2 \mbox{ smallest probabilities}].
\end{array}
\end{equation}
From the bounds \eqref{OS2 derivationaaaa} and  \eqref{OS2 derivationa}, we get:
\begin{equation*}
\begin{array}{rlll}   \label{OS2}
\overline{P}(n,k,\mb{p})
 \leq  \displaystyle
   \min\left(1,\min_{r_1 \in [0,k-1]} \left(\frac{S_{1r_1}}{k-r_1}\right), \min_{r_2 \in [0,k-2]}\left(\frac{S_{2r_2}}{\binom{k-r_2} {2}}\right)  \right),& \forall k \in [2,n],
\end{array}
\end{equation*}
where $S_{1r_1} = \sum_{i \in [n-r_1]}p_{i}$ for $r_1 \in [0,n-1]$ and $S_{2r_2} =  \sum_{(i,j) \in K_{n-r_2}} p_{i}p_{j}$ for $r_2 \in [0,n-2]$. One can interpret this bound as creating a set of dual feasible solutions and picking the best among them. The dual formulation is:
\begin{equation*}
\begin{array}{rlllll} \label{uniondual}
\displaystyle \overline{P}(n,k, \mb{p})  = \min & \displaystyle \sum_{(i,j) \in K_n} \lambda_{ij}p_{i}p_{j}+ \displaystyle \sum_{i \in [n]} \lambda_{i}p_{i} +\lambda_{0}  \\
 \mbox{s.t}  &\displaystyle \sum_{(i,j) \in K_n} \lambda_{ij}{c}_{i}{c}_{j}+ \displaystyle \sum_{i \in [n]} \lambda_{i}{c}_{i} +\lambda_{0} \geq 0 & \forall \mb{c} \in \{0,1\}^n,\\
 &\displaystyle \sum_{(i,j) \in K_n} \lambda_{ij}{c}_{i}{c}_{j}+ \displaystyle \sum_{i \in [n]} \lambda_{i}{c}_{i} +\lambda_{0} \geq 1, & \forall \mb{c} \in \{0,1\}^n: \sum_{t} {c}_{t} \geq k.
\end{array}
\end{equation*}
The components of the second term in \eqref{POSB1} are obtained by choosing dual feasible solutions with $\lambda_i = 1/(k-r_1)$  for $i \in [n-r_1]$ and setting all other dual variables to $0$. Similarly, the components of the third term are obtained by choosing dual feasible solutions with $\lambda_{ij} = 1/\binom{k-r_2} {2}$ for $(i,j) \in K_{n-r_2}$ and setting all other dual variables to $0$.
\texitem{(b)} The bound in \eqref{Prekopa3} is obtained by using the inequality:
\begin{equation*}
\begin{array}{lllll}
\displaystyle\mathbb{P}\left(\sum_{i \in [n]} \tilde{c}_i \geq k\right)   \leq   \displaystyle \mathbb{P}\left(\sum_{i \in [n-r]} \tilde{c}_{i} \geq k-r\right),  &   \forall r \in [0,k-1],
 \end{array}
\end{equation*}
in conjunction with the bound in \eqref{BorosPrekopa} computed from Boros and Pr\'{e}kopa \cite{boros1989}. We compute an upper bound on $\mathbb{P}\left(\sum_{i \in [n-r]} \tilde{c}_{i} \geq k-r\right)$ by using the aggregated moments $S_{1r}$ and $S_{2r}$ with the Boros and Pr\'{e}kopa bound from \eqref{BorosPrekopa} as follows:
\begin{equation*}
\begin{array}{lll}   \label{Prekopa2again}
BP(n-r,k-r,\mb{p}) \\
= \begin{cases}
 1,   &  \displaystyle k< \frac{(n-r-1)S_{1r}-2S_{2r}}{n-r-S_{1r}} +r, \\
 &\\
\displaystyle \frac{(k+n-2r-1)S_{1r}-2S_{2r}}{(k-r)(n-r)} ,& \displaystyle \frac{(n-r-1)S_{1r}-2S_{2r}}{n-r-S_{1r}}+r \leq k <1+\frac{2S_{2r}}{S_{1r}}+r,\\
 &\\
\displaystyle \frac{(i-1)(i-2S_{1r})+2S_{2r}}{(k-r-i)^2+(k-r-i)} ,&\displaystyle k \geq 1+\frac{2S_{2r}}{S_{1r}}+r, 
\end{cases}
\end{array}
\end{equation*}
where $i=\lceil{((k-r-1)S_{1r}-2S_{2r})}/({k-r-S_{1r}})\rceil$. Since the relation $\overline{P}(n,k,\mb{p}) \leq BP(n-r,k-r,\mb{p})$ is satisfied for every $0 \leq r \leq k-1$, the best upper bound on $\overline{P}(n,k,\mb{p})$ is obtained by taking the minimum over all possible values of $r$:
\begin{equation*}
\begin{array}{llll}
\displaystyle   \overline{P}(n,k,\mb{p})  \leq  \displaystyle {\min}_{r \in [0,k-1]} BP(n-r,k-r,\mb{p}),& \forall k \in [2,n].
\end{array}
\end{equation*}
\texitem{(c)} Proceeding in a similar manner as in (b), by using the aggregated moments $S_{1r}$ and $S_{2r}$ with Chebyshev bound, the upper bound for a given $r \in [0,k-1]$ can be written as follows:
\begin{equation*}
\begin{array}{lll}
CH(n-r,\; k-r,\;\mb{p}) = \begin{cases}
 1,   & k < S_{1r}+r,\\
 \displaystyle \frac{S_{1r}-(S_{1r}^2-2S_{2r})}{S_{1r}-(S_{1r}^2-2S_{2r}) +( k-r-S_{1r})^2} ,& S_{1r} +r \leq k \leq n.
  \end{cases}
\end{array}
\end{equation*}
The best upper bound on $\overline{P}(n,k,\mb{p})$ is obtained by taking the minimum over all possible values of $r$:
\begin{equation*}
\begin{array}{llll}
 \overline{P}(n,k,\mb{p})  \leq  \displaystyle \underset{r \in [0,k-1]} {\min}  CH(n-r,k-r,\mb{p}),& \forall k \in [2,n].
\end{array}
\end{equation*}
\end{proof}

\subsection{Connection to existing results}  Prior work in R\"{u}ger \cite{ruger1978} shows that ordering of probabilities provides the tightest upper bound on the probability of $n$ Bernoulli random variables adding up to at least $k$, when allowing for arbitrary dependence. Specifically, the bound derived there is:
$$\displaystyle  {\overline{P}}_u(n,k,\mb{p})  = \displaystyle \max_{\theta \in \Theta(\mbs{p})} \mathbb{P}_\theta\left( \sum_{i \in [n]} \tilde{c}_{i} \geq k\right)  = \min\left(1,\min_{r \in [0,k-1]} \left(\frac{S_{1r}}{k-r}\right) \right).$$
However, this bound does not use pairwise independence information. Part (a) of Theorem \ref{thm:nonidentical} tightens the analysis in R\"{u}ger \cite{ruger1978} for pairwise independent random variables. It is also straightforward to see that the ordered Schmidt, Siegel and Srinivasan bound in \eqref{POSB1} is at least as good as the bound in \eqref{SSS} (simply plug in $r = 0$). Building on the ordering of probabilities, the bound in \eqref{Prekopa3} uses aggregated binomial moments for $k$ ordered sets of random variables of size $n-r$ where $r \in [0,k-1]$. When $r = 0$, the bound in \eqref{Prekopa3} reduces to the original aggregated moment bound of Boros and Pr\'{e}kopa in \eqref{BorosPrekopa}
and hence this bound is at least as tight. All the bounds in Theorem \ref{thm:nonidentical} are clearly efficiently computable.

\noindent It is easy to verify that the ordered Boros and Pr\'{e}kopa bound is at least as good as the other two ordered bounds, \emph{i.e.},
\begin{equation*}
\begin{array}{lllllllll}
 \displaystyle \mbox{Ordered bound }  \eqref{Prekopa3}  \leq  \displaystyle \min\left(\mbox{Ordered bound } \eqref{POSB1},\mbox{Ordered bound }\eqref{chebyffinal}\right).
\end{array}
\end{equation*}
This is true since, each term of the ordered bounds are derived by finding upper bounds on the probability that the sum of the first $n-r$ random variables takes a value of at least $k-r$ using only the first two moments of the sum of these random variables. Since the Boros and Pr\'{e}kopa bound is the tightest upper bound possible when using only the first two moments of the sum, each term in the ordered Boros and Pr\'{e}kopa bound is at least as good as the corresponding term in the other two ordered bounds. Taking the minimum over all these terms implies that the ordered Boros and Pr\'{e}kopa bound must be at least as good as the other two bounds.
\subsection{Further tightening of ordered bounds:}\label{subsec:furthertight}
\noindent It is also worth mentioning that the bounds in Theorem \ref{thm:nonidentical} can in fact be strengthened further by using the tightest possible bound for $k = 1$ from Theorem \ref{thm:unionbound}. Specifically, we can tighten the ordered Schmidt, Siegel and Srinivasan bound in \eqref{POSB1} as follows:
\begin{equation*}
\begin{array}{rllll}
\displaystyle
   \min\left(1,\min_{r \in [0,k-2]}\min\left(\frac{S_{1r}}{k-r},\frac{S_{2r}}{\binom{k-r} {2}}\right),\sum_{i \in [n-k+1]}p_i-p_{n-k+1}\sum_{i \in [n-k]}p_i  \right).
\end{array}
\end{equation*}
where the last term corresponds to $r_1=k-1$ and is obtained by observing that:
\begin{equation*}
\begin{array}{llll}
\displaystyle \mathbb{P}\left(\sum_{i \in [n]} \tilde{c}_i \geq k\right)  & \leq & \displaystyle \mathbb{P}\left(\sum_{i \in [n-k+1]} \tilde{c}_{i} \geq 1\right) \\
& \leq & \displaystyle \sum_{i \in [n-k+1]}p_i-p_{n-k+1}\sum_{i \in [n-k]}p_i \\
&& [\mbox{from Theorem }\ref{thm:unionbound}]\\
& \leq &  \sum_{i \in [n-k+1]}p_i \\
& = & S_{1(k-1)}/(k-(k-1)).
\end{array}
\end{equation*}
The Boros and Pr\'{e}kopa bound and Chebyshev ordered bounds in   \eqref{Prekopa3}  and \eqref{chebyffinal} can be similarly tightened. Unlike the bounds in Theorem \ref{thm:nonidentical}, these tightened bounds use partially disaggregated moment information. We next provide two numerical examples to illustrate the impact of ordering on the quality of the three bounds. We restrict attention, however, to the aggregated ordered moment bounds in Theorem \ref{thm:nonidentical} only.
\subsection{Numerical illustrations}\label{subsec:numericals_noniden}
\begin{example} [Non-identical marginals] \ Consider an example with $n = 12$ random variables with the probabilities given by
\begin{equation*}
\begin{array}{lllllllll}
p_1= 0.0651, p_2= 0.0977, p_3 = 0.1220, p_4 = 0.1705, p_5= 0.3046, p_6 = 0.4402,\\
p_7 = 0.4952, p_8 = 0.6075, p_9 = 0.6842, p_{10} = 0.8084, p_{11} = 0.9489, p_{12} = 0.9656.
\end{array}
\end{equation*} Table \ref{tab:heteroall3vstight} compares the three ordered bounds with the three unordered bounds and the tight upper bound. Numerically, the ordered Boros and Pr\'{e}kopa bound \eqref{Prekopa3} is found to be tight in this example for $k=7,8,9,12$ while the ordered Schmidt, Siegel and Srinivasan bound \eqref{POSB1} is tight for $k=12$. The ordered Boros and Pr\'{e}kopa bound is uniformly the best performing of the three bounds, while among the other two ordered bounds, none uniformly dominates the other. For example, comparing the ordered bounds when $7 \leq k \leq 9$, the Chebyshev bound outperforms the Schmidt, Siegel and Srinivasan bound, but when $k=6$ or $10 \leq k \leq 12$, the Schmidt, Siegel and Srinivasan bound does better. Comparing the unordered bounds when $7 \leq k \leq 9$, the Schmidt, Siegel and Srinivasan bound \eqref{SSS} outperforms the Chebyshev bound \eqref{cheby} when $k=6$ but for all $k\geq 7$, bound \eqref{cheby} does better. In terms of absolute difference between ordered and unordered bounds, ordering provides the maximum improvement to the Schmidt, Siegel and Srinivasan bound, followed by the Boros and Pr\'{e}kopa bound and the Chebyshev bound.\\

\begin{table}[H]
\footnotesize
\caption{Upper bound on the probability of sum of random variables equaling at least $k$ for $n = 12$. For each value of $k$, the bottom row provides the tightest bound which can be computed in this example by solving an exponential sized linear program. The underlined instances illustrate cases when the other upper bounds are tight.} \label{tab:heteroall3vstight}
\begin{center}
\footnotesize{\begin{tabular}
{|l|c|c|l|l|l|l|l|l|l|}
  \hline
\mbox{Bound} & $k \in [1,4]$ & $k = 5$ & $k = 6$ & $k = 7$ & $k = 8$ & $k = 9$ & $k = 10$ & $k = 11$ & $k = 12$\\ \hline
\eqref{cheby}   & 1 & 1 &  0.9553 &  0.5192 & 0.2552 & 0.1424 & 0.0889 & 0.0603 & 0.0434 \\
 \hline
\eqref{chebyffinal} & 1 & 1 &  0.9553 &  0.5192 & 0.2552 & 0.1424 & 0.0883 & 0.0549 & 0.0307 \\
 \hline
  \eqref{SSS}&
   1&1&   0.9517 & 0.6831 & 0.5123& 0.3985 & 0.3188 &  0.2608 & 0.2173\\
 \hline
  \eqref{POSB1}  &
   1&1&   0.9489 & 0.6162 & 0.3620& 0.1827 & 0.0712 &  0.0250 & \underline{0.0064}\\
 \hline
  \eqref{BorosPrekopa} &
 1 & 1 &  0.9497 &  \underline{0.5018} & \underline{0.2509} & 0.1326 & 0.0795 & 0.0530 & 0.0379\\
 \hline
 \eqref{Prekopa3} &
 1 & 1 &  0.9254 &  \underline{0.5018} & \underline{0.2509} & \underline{0.1290} & 0.0712 & 0.0249 & \underline{0.0064}\\
  \hline
 \mbox{Tight}  &
   \centering {1} &   0.9957 & 0.8931 & 0.5018 & 0.2509 & 0.1290 & 0.0692 & 0.0230 &  0.0064 \\
  \hline
\end{tabular}}
\end{center}
\end{table}
\end{example}
\begin{example} [Non-identical marginals]
In this example, we numerically compute the improvement of the new ordered bounds over the unordered bounds for $n=100$ variables by creating $500$ instances by randomly generating the probabilities $\mb {p}=(p_1,p_2,..,p_{100})$. First, we consider small marginal probabilities by uniformly and independently generating the entries of $\mb{p}$ between $0.01$ and $0.05$. When $k=n$, Figure \ref{fig:hetero3a} plots the three ordered bounds while Figure \ref{fig:hetero3b} shows the percentage improvement of the three bounds over their unordered counterparts. The percentage improvement is computed as \big([unordered-ordered]/unordered\big)$\times$ 100$\%$. In this example with small marginals, the ordered Schmidt, Siegel and Srinivasan bound \eqref{POSB1} is equal to the ordered Boros and Pr\'{e}kopa bound \eqref{Prekopa3} as seen in Figure \ref{fig:hetero3a}.  Ordering tends to improve the Schmidt, Siegel and Srinivasan bound significantly for smaller probabilities, since both the partial binomial moment terms $S_{1r}$ and $S_{2r}$ are smaller with smaller marginal probabilities for all $r \in [0,k-1]$.

\begin{figure}[htbp]
\begin{subfigure}[b]{0.45\linewidth}
\includegraphics[scale=0.4]{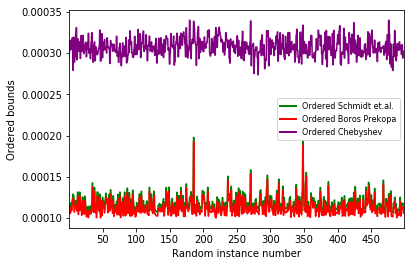}
\caption{Actual value of the ordered bounds}
\footnotesize \label{fig:hetero3a}
\end{subfigure}
\begin{subfigure}[b]{0.45\linewidth}
\includegraphics[scale=0.4]{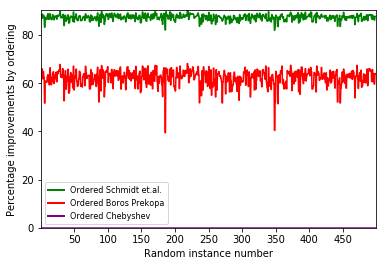}
\caption{Percentage improvement of ordered bounds}
\footnotesize \label{fig:hetero3b}
\end{subfigure}%
\caption{Smaller marginal probabilities $p_i$ with $n=100,k=100$ and $500$ instances.}
\footnotesize \label{fig:hetero3c}
\end{figure}
 The percentage improvement due to ordering in figure \ref{fig:hetero3b} is consistently above $80\%$ for the Schmidt, Siegel and Srinivasan bound, while that of the Boros and Pr\'{e}kopa bound is around $60 \%$. The ordered Chebyshev bound \eqref{chebyffinal} shows an almost negligible improvement by ordering in this example.

Next, we consider similar plots when $k=n-1$ with larger marginal probabilities. The entries of $\mb{p}$ are generated uniformly and independently between $0.05$ and $0.99$.
\begin{figure}[htbp]
\begin{subfigure}[b]{0.45\linewidth}
\includegraphics[scale=0.4]{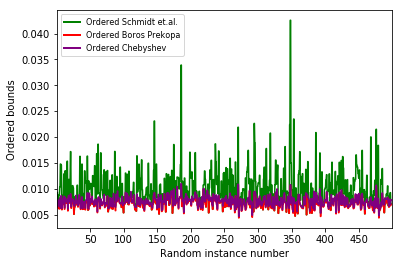}
\caption{Actual value of the ordered bounds}
\label{fig:hetero4a}
\end{subfigure}
\begin{subfigure}[b]{0.45\linewidth}
\includegraphics[scale=0.4]{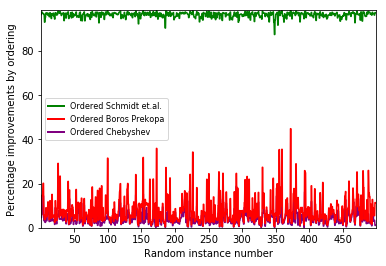}
\caption{Percentage improvement of ordered bounds}
\label{fig:hetero4b}
\end{subfigure}%
\caption{Larger marginal probabilities $p_i$ with $n=100,k=99$ and $500$ instances.}
\label{fig:hetero4c}
\end{figure}
In Figure \ref{fig:hetero4a}, the ordered Chebyshev bound \eqref{chebyffinal} performs better than the ordered Schmidt, Siegel and Srinivasan bound \eqref{POSB1}. In Figure \ref{fig:hetero4b}, the  percentage improvement due to ordering is again most significant for the Schmidt, Siegel and Srinivasan bound, being consistently above $90\%$ while that of the Boros and Pr\'{e}kopa bound is less than $40 \%$ and that of the Chebyshev bound is less than $20 \%$.
It is also clear from Figures \ref{fig:hetero3c} and \ref{fig:hetero4c} that the ordered Boros and Pr\'{e}kopa bound \eqref{Prekopa3} is the tightest of the three bounds across the instances, while among the other two bounds, none uniformly dominates the other.
\end{example}

\section{Tightness in special cases}\label{sec:tightinstances}
In this section, we identify two tight instances, one for the unordered bounds in \eqref{cheby}, \eqref{SSS} and \eqref{BorosPrekopa} and the other for the corresponding ordered bounds derived in Theorem \ref{thm:nonidentical}. Firstly, in Section \ref{subsec:idenBP}, for identical variables, the symmetry in the problem allows for closed-form tight bounds for any $k \in [2,n]$. We prove this by showing an equivalence of the exponential sized linear program \eqref{eq:generalprimal} which computes the exact bound with a polynomial sized linear program analyzed in computing the Boros and Pr\'{e}kopa bound in \eqref{BorosPrekopa}. We use the exact bound to identify instances when the other two unordered bounds are tight. The result with identical marginals is further extended to show tightness for $t$-wise independent variables. Secondly, in Section \ref{subsec:almostidentight}, we demonstrate the usefulness of the ordered bounds by identifying a special case when $n-1$ marginals are identical (with additional conditions on the probability and $k$), when the ordered bounds in \eqref{POSB1} and \eqref{Prekopa3} are tight.

\subsection{Tightness of bounds with identical marginals} \label{subsec:idenBP}
In this section, we provide probability bounds for $n$ pairwise independent random variables adding up to at least $k \in [2,n]$ when their marginals are identical.
The next theorem provides the tight bound with identical marginals, by applying the Boros and Pr\'{e}kopa bound in \eqref{BorosPrekopa} to pairwise independent variables with $\tilde{\xi} =\sum_{i \in [n]}{\tilde{c}}_i$.
\begin{theorem} \label{thm:idenBPtight}
Assume $p_{i}=p \in (0,1)$ for $i \in [n]$. Let $\overline{P}(n,k,p)$ represent the tightest upper bound on the probability that $n$ pairwise independent identical Bernoulli random variables add up to at least $k \in [n]$. Then,
\begin{equation}
 \label{BP>=iden}
\overline{P}(n,k,p) = \left\{
\begin{array}{lll}
 1,   &  k<(n-1)p, & \mbox{(a)}\\
 \displaystyle\frac{((n-1)(1-p)+k)p}{k} ,&(n-1)p \leq k < 1+(n-1)p, & \mbox{(b)} \\
\displaystyle\frac{(i-1)(i-2np)+n(n-1)p^2}{(k-i)^2+(k-i)} ,& k \geq 1+(n-1)p,& \mbox{(c)},
 \end{array}
 \right.
\end{equation}
where $i=\lceil{np(k-1-(n-1)p)}/{(k-np)}\rceil$.
\end{theorem}

\begin{proof} The tightest upper bound $ \overline{P}(n,k,p)$ is the optimal value of the linear program:
\begin{equation}
\begin{array}{rlllll} \label{>=kpweight1}
\displaystyle \overline{P}(n,k,p)  = \max & \displaystyle \sum_{\mbs{c} \in \{0,1\}^n: \sum_{i} {c}_{i} \geq k} \theta(\mb{c}) \\
 \mbox{s.t}  &\displaystyle  \sum_{\mbs{c} \in \{0,1\}^n} \theta(\mb{c})= 1,   \\
 & \displaystyle \sum_{\mbs{c} \in \{0,1\}^n: c_i = 1} \theta(\mb{c})= p , & \forall i \in [n], \\
 & \displaystyle \sum_{\mbs{c} \in \{0,1\}^n: c_i = 1, c_j = 1}\theta(\mb{c})= p^2, & \forall (i,j) \in K_n,  \\
  & \displaystyle \theta(\mb{c}) \geq 0, & \forall \mb{c} \in \{0,1\}^n,
\end{array}
\end{equation}
where the decision variables are the joint probabilities $\theta(\mb{c}) = \mathbb{P}(\tilde{\mb{c}} = \mb{c})$ for $\mb{c} \in \{0,1\}^n$. Consider the following linear program in $n+1$ variables which provides an upper bound on $\overline{P}(n,k,p)$:
\begin{equation}
\begin{array}{rllll} \label{BP>=}
\displaystyle BP(n,k,p)  = \max & \displaystyle \sum_{\ell \in [k,n]}  v_{\ell}\\
\mbox{s.t.}  & \displaystyle \sum_{\ell \in [0,n]} v_{\ell}= 1,  \\
& \displaystyle \sum_{\ell \in [1,n]} \ell v_{\ell}=np,  \\
 & \displaystyle\sum_{\ell \in [2,n]} \binom{\ell}{2}v_{\ell} =\binom{n}{2}p^2, \\
  & \displaystyle v_{\ell} \geq 0, & \forall \ell \in [0,n],
\end{array}
\end{equation}
 where the decision variables are the probabilities $v_{\ell}=\mathbb{P}(\sum_{i \in [n]} \tilde{c}_i = \ell)$ for $l \in [0,n]$. Linear programs of the form (\ref{BP>=}) have been studied in Boros and Pr\'{e}kopa \cite{boros1989} in the context of aggregated binomial moment problems. As we shall see, these two formulations are equivalent with identical pairwise independent random variables.
\texitem{(1)} $\overline{P}(n,k,p) \leq BP(n,k,p)$: Given a feasible solution to \eqref{>=kpweight1} denoted by $\theta$, construct a feasible solution to the linear program \eqref{BP>=} as:
\begin{align*}
v_{\ell}=\sum_{\mbs{c} \in \{0,1\}^n: \sum_{i} {c}_{i} = l} \theta(\mb{c}), \quad \forall l \in [0,n].
\end{align*}
By taking expectations on both sides of the equality \eqref{binomial}, we get:
\begin{align*}
\displaystyle \sum_{l \in [j,n]} \binom {l}{j} \mathbb{P}\left(\sum_{i \in [n]}\tilde{c}_i = l\right)=\mathbb{E}\left[S_{j}(\tilde{\mb{c}})\right],\quad \forall j \in [0,n].
\end{align*}
Applying it for $j= 0,1,2$, we get the three equality constraints in \eqref{BP>=}:
\begin{equation*}
\begin{array}{llll}
\displaystyle \sum_{\ell \in [0,n]} v_{\ell} = 1,\\
\displaystyle \sum_{\ell \in [1,n]} \ell v_{\ell}=\mathbb{E}\left[\displaystyle \sum_{i \in [n]} \tilde{c}_{i}\right] = np,\\
\displaystyle \sum_{\ell \in [2,n]} \binom{\ell}{2}v_{\ell}=\mathbb{E}\left[\displaystyle \sum_{(i,j) \in K_n} \tilde{c}_{i}\tilde{c}_{j}\right] = n(n-1)p^2/2.
\end{array}
\end{equation*}
Lastly, the objective function value of this feasible solution satisfies:
\begin{equation*}
\begin{array}{llll}
\displaystyle \sum_{\ell=k}^{n}v_{\ell} & = & \displaystyle \sum_{\ell=k}^{n}\; \sum_{\;\mbs{c} \in \{0,1\}^n: \sum_{i} {c}_{i} = l} \theta(\mb{c})\\
& = & \displaystyle \sum_{\mbs{c} \in \{0,1\}^n: \sum_{i} {c}_{i} \geq k} \theta(\mb{c}).
\end{array}
\end{equation*}
Hence, $\overline{P}(n,k,p) \leq BP(n,k,p)$.
\texitem{(2)} $\overline{P}(n,k,p) \geq BP(n,k,p)$: Given an optimal solution to \eqref{BP>=} denoted by $\mb{v}$, construct a feasible solution to the linear program  \eqref{>=kpweight1} by distributing $v_{\ell}$ equally among all the realizations in $\{0,1\}^n$ with exactly $\ell$ ones:
\begin{equation*}
\begin{array}{lll}
\displaystyle \theta(\mb{c})= \frac{v_{\ell}}{\binom{n}{\ell}}, & \forall \mb{c} \in \{0,1\}^n: \sum_{i \in [n]}  {c}_i =\ell, \forall \ell \in [0,n].
 \end{array}
\end{equation*}
The first constraint in \eqref{>=kpweight1} is satisfied since:
\begin{equation*}
\begin{array}{llll}
\displaystyle \sum_{\mbs{c} \in \{0,1\}^n} \theta(\mb{c}) & = & \displaystyle \sum_{\ell \in [0,n]}\sum_{\mbs{c} \in \{0,1\}^n: \sum_{i} {c}_{i} = l} \frac{v_{\ell}}{\binom{n}{\ell}}\\
& & [\mbox{since } \big\vert{\{0,1\}^n: \sum_{i \in [n]}  {c}_i =\ell }\big\vert=\binom{n}{\ell}] \\
& = & \displaystyle \sum_{\ell \in [0,n]} v_{\ell} \\
& = & \displaystyle 1.
\end{array}
\end{equation*}
The second constraint in \eqref{>=kpweight1} is satisfied since:
\begin{equation*}
\begin{array}{llllll} \label{>=kpweight2}
\displaystyle \sum_{\mbs{c} \in \{0,1\}^n: c_j = 1} \theta(\mb{c}) &=&\displaystyle \sum_{\ell \in [1,n]} \frac{v_{\ell}}{\binom{n}{\ell}}\binom{n-1}{\ell-1}\\
& & [\mbox{since } \big\vert{\{0,1\}^n: \sum_{i \in [n]}  {c}_i =\ell, {c}_j=1 }\big\vert=\binom{n-1}{\ell-1}] \\
& =&\displaystyle \sum_{\ell \in [1,n]} \frac{\ell v_{\ell}}{n}\\
& =& \displaystyle p.
\end{array}
\end{equation*}
The third constraint in \eqref{>=kpweight1} satisfied since:
\begin{equation*}
\begin{array}{llllll} \label{>=kpweight3}
\displaystyle \sum_{\mbs{c} \in \{0,1\}^n: c_i = 1, {c}_{j} =1} \theta(\mb{c}) &=&\displaystyle \sum_{\ell \in [2,n]} \frac{v_{\ell}}{\binom{n}{\ell}} \binom{n-2}{\ell-2} \\
& & [\mbox{since } \big\vert{\{0,1\}^n: \sum_{t \in [n]}  {c}_t =\ell, {c}_i=1,{c}_j=1 }\big\vert=\binom{n-2}{\ell-2}] \\
& = & \displaystyle \frac{2}{n(n-1)}\sum_{\ell \in [2,n]} \binom{\ell}{2}v_{\ell}  \\
&= & \displaystyle p^2.
\end{array}
\end{equation*}
The objective function value of the feasible solution is given by:
\begin{equation*}
\begin{array}{llllll}
\displaystyle \sum_{\mbs{c} \in \{0,1\}^n: \sum_{i} {c}_{i} \geq k} \theta(\mb{c}) &=& \displaystyle \sum_{\ell \in [k,n]}\sum_{\mbs{c} \in \{0,1\}^n: \sum_{i} {c}_{i} = l} \theta(\mb{c})
\\
&= & \displaystyle \sum_{\ell \in [k,n]}v_{\ell}\\
& =& BP(n,k,p).
 \end{array}
 \end{equation*}
Hence, the optimal objective value of the two linear programs are equivalent. The formula for the tight bound in the theorem is then exactly the Boros and Pr\'{e}kopa bound in \eqref{BorosPrekopa} (the bound $BP(n,k,p)$ is also derived in the work of \cite{sathe1980}, although tightness of the bound is not shown there). It is straightforward to verify that the following distributions attain the bounds for each of the cases (a)-(c) in the statement of the theorem:
\texitem{(a)}
The probabilities are given as:
\begin{equation*} \label{d1}
\displaystyle \theta(\mb{c}) = \left\{\begin{array}{llr}
\displaystyle \frac{(1-p)(j-(n-1)p)}{\binom{n-1}{j-1}}, & \mbox{if } \displaystyle \sum_{t \in [n]} {c}_{t} = j-1, \\
\displaystyle \frac{(1-p)(1+(n-1)p-j)}{\binom{n-1}{j}}, & \mbox{if } \displaystyle \sum_{t \in [n]} {c}_{t} = j, \\
\displaystyle \frac{n(n-1)p^2+(j-1)(j-2np)}{(n-j)^2+(n-j)} ,& \mbox{if } \displaystyle \sum_{t \in [n]} {c}_{t} = n,
\end{array}\right.
\end{equation*}
where $j=\ceil{(n-1)p}$ and all other support points have zero probability.
\texitem{(b)} The probabilities are given as:
\begin{equation*} \label{d2}
\displaystyle \theta(\mb{c}) = \left\{\begin{array}{llr}
\displaystyle \frac{1-p}{k}(k-(n-1)p), & \mbox{if } \displaystyle \sum_{t \in [n]} {c}_{t} = 0,\\
\displaystyle \frac{p(1-p)}{\binom{n-2}{k-1}}, & \mbox{if } \displaystyle \sum_{t \in [n]} {c}_{t} =k, \\
\displaystyle \frac{p((n-1)p-(k-1))}{n-k}, & \mbox{if } \displaystyle \sum_{t \in [n]} {c}_{t} = n,
\end{array}\right.
\end{equation*}
where all other support points have zero probability.
\texitem{(c)} The probabilities are given as:
\begin{equation*} \label{d3}
\displaystyle \theta(\mb{c}) = \left\{\begin{array}{llr}
\displaystyle \frac{np[(n-1)p-(k+i-1)]+ik}{\binom{n}{i-1}(k-i+1)}, & \mbox{if } \displaystyle \sum_{t \in [n]} {c}_{t} = i-1,\\
\displaystyle \frac{np[(k+i-2)-(n-1)p]-k(i-1)}{\binom{n}{i}(k-i)}, & \mbox{if } \displaystyle \sum_{t \in [n]} {c}_{t} =i, \\
\displaystyle \frac{n(n-1)p^2+(i-1)(i-2np)}{\binom{n}{k}[(k-i)^2+(k-i)]}, & \mbox{if } \displaystyle \sum_{t \in [n]} {c}_{t} = k,
\end{array}\right.
\end{equation*}
where all other support points have zero probability and the index $i$ is evaluated as stated in equation (\ref{BP>=iden})(c).
It is straightforward to see that with identical marginals, the tight union bound in Theorem \ref{thm:unionbound} reduces to the bound in case (b) of Theorem \ref{thm:idenBPtight}.
 \end{proof}

\subsubsection{Connection of Theorem \ref{thm:idenBPtight} to existing results} Tightness results with identical Bernoulli random variables have been established in the literature in the context of occurence of at least and exactly $k$ out of $n$ events for specific regimes of the parameters $n,k$ and $p$. Theorem \ref{thm:idenBPtight} however, provides the tight bounds for all values of $(n,k,p)$. Recent work by Garnett \cite{garnett} provides the tight upper bound on the probability that the sum of pairwise independent Bernoulli random variables exceeds the mean by a small amount (this corresponds to case (b)). Pinelis \cite{pinelis2021exact} derives a closed-form tight lower bound on the probability of occurence of exactly one of out $n$ events. Benjamini et al. \cite{benjamini12} and Peled et al. \cite{peled} derived closed-form upper and lower bounds (not necessarily tight) on the maximal intersection probability of more general $t$-wise independent Bernoulli random variables (this corresponds to $k = n$ in case (c) for $t=2$). These bounds were shown to match each other up to multiplicative factors of lower order in a large regime of the parameters $n,p,t$. The connection of the intersection probability with the linear program based approach of Boros and Pr\'{e}kopa \cite{boros1989} has been mentioned in these papers, although the equivalence for all values of $k$ is not established. Corollary \ref{cor:twiseiden} in this paper, however, establishes the equivalence for all values of $n,k,p,t$.
The usefulness of Theorem \ref{thm:idenBPtight} lies in the fact that it can be extended to incorporate a wide variety of cases involving identical Bernoulli events by using the results from Boros and Pr\'{e}kopa \cite{boros1989} as follows:
 \begin{enumerate}[label=\roman*),wide=0pt]
  \item Tight closed-form lower bounds on probability of occurrence of at least $k$ out of $n$ events

  \item  Tight closed-form upper and lower bounds on the probability of occurence of exactly $k$ out of $n$ events

  \item Tight linear program based upper and lower bounds for $t$-wise independent variables ($t \geq 3$) from the symmetry assumptions (see Corollary \ref{cor:twiseiden}).
  \end{enumerate}
We note that when $k\geq 1+(n-1)p$, the tight lower bound from \cite{boros1989} can be derived as:
\begin{equation*}
 \label{minBP>=iden}
\scalebox{1}{$\underline{P}(n,k,p)$} = \left\{
\scalebox{1}{$
\begin{array}{lll}
\scalebox{1}{$\frac{\big( 2+(n-1)p-k\big)p}{n-k+1} $},& 1+(n-1)p\leq k<2+(n-1)p \\
0,& k\geq2+(n-1)p.
 \end{array}$}
 \right.
\end{equation*}
When $k=n \geq 1+(n-1)p$, this bound reduces to $\max\left(p\left((n-1)p-(n-2)\right),0\right)$ which is exactly the intersection bound computed in Corollary \ref{cor:intersection} with identical probabilities.
\begin{corollary}\label{cor:twiseiden}
Consider identical $t$-wise independent Bernoulli random variables with probabilities $p \in (0,1)$ where $t \in [2,n]$. Then, the tightest upper bound on the probability of $n$ such variables adding up to at least $k \in [n]$, denoted by $\overline{P}(n,k,p,t)$, can be computed as the optimal value of the aggregated linear program proposed by Pr\'{e}kopa \cite{prekopa1990}:
\begin{equation} \label{opt:twiseBP>=iden}
\begin{array}{rllll}
\overline{P}(n,k,p,t)  = \max & \displaystyle \sum_{\ell=k}^n  v_{\ell}\\
 \textrm{s.t.}   & \displaystyle\sum_{\ell=m}^n \binom{\ell}{m}v_{\ell} =\binom{n}{m}p^m, & \forall m \in [0,t],\\
  & \displaystyle v_{\ell} \geq 0, &\forall \ell \in [0,n],
\end{array}
\end{equation}
\end{corollary}where the decision variables are the probabilities $v_{\ell}=\mathbb{P}(\sum_{i=1}^n \tilde{c}_i = \ell)$ for $l \in [0,n]$.

 \begin{proof}
 The proof is straightforward from the proof of Theorem \ref{thm:idenBPtight} which implies the equivalence of \eqref{opt:twiseBP>=iden} with the large-sized linear program:
 \begin{equation}\label{opt:twiseexpo}
\begin{array}{rrllll}
\displaystyle \overline{P}(n,k,p,t)  =\max & \displaystyle \sum_{\mbs{c} \in \{0,1\}^n: \sum_{i} {c}_{i} \geq k} \mathbb{P}(\mb{c}) \\
 \textrm{s.t.}  &\displaystyle  \sum_{\mbs{c} \in \{0,1\}^n} \mathbb{P}(\mb{c})&=& 1,  \\
 & \displaystyle \sum_{\mbs{c} \in \{0,1\}^n:\; c_{i} =1,  \;\forall i \in  J}\mathbb{P}(\mb{c})&=& p^m, &  \forall J \in I_m, \;m \in [t], \\
  & \displaystyle \mathbb{P}(\mb{c})& \geq &0, & \forall \mb{c} \in \{0,1\}^n,
\end{array}
\end{equation}
where $ I_m=\{I \subseteq [n]: |I|=m\}$. In particular for any given feasible solution of \eqref{opt:twiseBP>=iden}, we can distribute the probability mass $v_{\ell}$ evenly across the $\binom{n}{\ell}$ scenarios for every $\ell \in [0,n]$ and satisfy all the constraints in \eqref{opt:twiseexpo} while for any given feasible solution of \eqref{opt:twiseexpo}, we can aggregate the probabilities $\mathbb{P}(\mb{c})$ as
\begin{align*}
v_{\ell}=\sum_{\mbs{c} \in \{0,1\}^n: \sum_{i} {c}_{i} = l} \mathbb{P}(\mb{c}), \quad \forall l \in [0,n].
\end{align*} and satisfy all constraints in  \eqref{opt:twiseBP>=iden}.
 \end{proof}
\noindent We note that for $3$-wise independent variables, a closed-form expression for the optimal objective in \eqref{opt:twiseBP>=iden} using the first three binomial moments has been provided in \cite{boros1989}. Further, the corresponding tight lower bound $\underline{P}(n,k,p,t)$ can be computed as the optimal value of the minimization version of the aggregated linear program in \eqref{opt:twiseBP>=iden}.

\subsubsection{Tightness of alternative bounds}\label{subsubsec:identightalternative}
We next discuss an application of Theorem \ref{thm:idenBPtight}. Since the marginals are identical, it is easy to see that the ordered bounds in Theorem \ref{thm:nonidentical} reduce to the unordered bounds corresponding to $r=0$. While the unordered Boros and Pr\'{e}kopa bound provides the tightest upper bound with identical marginals, the formula is more involved than the unordered Chebyshev bound which reduces to:
\begin{equation} \label{eq:ChebyH}
\begin{array}{lll}
\overline{P}(n,k,{p}) \leq \begin{cases}
 1,   & k < np, \\
  { \displaystyle np(1-p) }/{\displaystyle \left(np(1-p) +(k- np)^2\right)} ,& np \leq k \leq n.
  \end{cases}
\end{array}
\end{equation}
and the unordered Schmidt, Siegel and Srinivasan bound which reduces to:
\begin{equation}  \label{eq:SSSH}
\begin{array}{lll}
\overline{P}(n,k,p) \leq
 \min\bigg(1, \displaystyle\frac{np}{k} ,\frac{n(n-1)p^2}{k(k-1)}\bigg).
\end{array}
\end{equation}
It is possible to then use Theorem \ref{thm:idenBPtight} to identify conditions on the parameters $(n,k,p)$ for which the bounds in \eqref{eq:ChebyH} and \eqref{eq:SSSH} are tight. We only focus on the non-trivial cases where the tight bound is strictly less than one and $n \geq 3$. Henceforth, the Chebyshev and Schmidt, Siegel and Srinivasan bounds referred to in this section are the unordered bounds.
\begin{prop}\label{prop:tighthom}\
 \texitem{(a)} For $p=\alpha/(n-1)$ and any integer $\alpha \in [n-2]$, the Chebyshev bound in \eqref{eq:ChebyH} is tight for the values of $k = \alpha+1$ and $k = n$.
\texitem{(b)} For $p\leq 1/(n-1)$, the Schmidt, Siegel and Srinivasan bound in \eqref{eq:SSSH} is tight for all $k \in [2,n]$ while for $p>1/(n-1)$, the bound is tight for all values of $k \in [\left\lceil 1+(n-1)p \right\rceil, \left\lfloor n(n-1)p^2/(np-1) \right\rfloor]$.
\end{prop}

\begin{proof}
Since Theorem \ref{thm:idenBPtight} provides the tight bound, we simply need to show the equivalence with the bounds in \eqref{eq:ChebyH} and \eqref{eq:SSSH} for the instances in the proposition.
 \texitem{(a)}
Consider $p=\alpha/(n-1)$ for any integer $\alpha \in [n-2]$.
  \begin{enumerate}
  \item Set $k= \alpha +1$. This corresponds to case (c) in Theorem \ref{thm:idenBPtight}. Plugging in the values, the index $i$ which is required for finding the tight bound is given by:
   \begin{equation*}
\begin{array}{lll}
i  &=&\displaystyle \left\lceil\frac{n\alpha(\alpha+1-1-\alpha)/(n-1)}{\alpha+1-n\alpha/(n-1)}\right\rceil\\
& = & 0.
    \end{array}
 \end{equation*}
 The corresponding tight bound in \eqref{BP>=iden} gives:
  \begin{equation*}
\begin{array}{lll}   \label{chebysymmetric11}
 \overline{P}(n,k,p) =& \displaystyle \frac{n\alpha}{(n-1)(\alpha+1)}=& \displaystyle \frac{np}{np+1-p}.
  \end{array}
 \end{equation*}
 It is straightforward to verify by plugging in the values that the Chebyshev bound is exactly the same.
 \item Set $k= n$. This corresponds to case (c) in Theorem \ref{thm:idenBPtight}. Plugging in the values, the index $i$ in the tight bound is given by:
   \begin{equation*}
\begin{array}{lll}
i  &=&\displaystyle \left\lceil\frac{n\alpha(n-1-\alpha)/(n-1)}{n-n\alpha/(n-1)}\right\rceil\\
& = & \alpha.
    \end{array}
 \end{equation*}
 The tight bound in \eqref{BP>=iden} gives:
  \begin{equation*}
\begin{array}{lll}   \label{chebysymmetric111}
 \overline{P}(n,k,p) =& \displaystyle \frac{\alpha}{(n-1)(n-\alpha)}=& \displaystyle \frac{p}{p+n(1-p)}.
  \end{array}
 \end{equation*}
 It is straightforward to verify by plugging in the values that the Chebyshev bound is exactly the same in this case.
\end{enumerate}
\texitem{(b)}
Observe that the last two terms in the Schmidt, Siegel and Srinivasan bound in \eqref{eq:SSSH} satisfy:
\begin{align*}
\frac{n(n-1)p^2}{k(k-1)} \leq \frac{ np}{k} \mbox{ when } k \geq 1+(n-1)p.
\end{align*}
Since $k \geq 1+(n-1)p$ implies $1 \geq np/k$, the bound in \eqref{eq:SSSH} reduces to ${n(n-1)p^2}/{k(k-1)}$. The range of $k \geq 1+(n-1)p$ corresponds to case (c) in Theorem \ref{thm:idenBPtight}. If $k=1+(n-1)p$, the index   $i =\lceil {np(k-(1+(n-1)p))}/{(k-np)}\rceil=0$ and the tight bound from \eqref{BP>=iden} is:
$$\displaystyle \frac{ np}{1+(n-1)p},$$
which is exactly the Schmidt, Siegel and Srinivasan bound. We can also verify that when the index $i = 1$ in case (c), then the tight bound in \eqref{BP>=iden} reduces to:
   \begin{equation*}
\begin{array}{lll}
 \overline{P}(n,k,p) &=& \displaystyle \frac{n(n-1)p^2+(1-1)(1-2np)}{(k-1)^2+(k-1)} \\
& = &  \displaystyle \frac{n(n-1)p^2}{k(k-1)}.
   \end{array}
 \end{equation*}
We now identify conditions when $k>1+(n-1)p$ and the index $i$ is equal to one.
 \begin{enumerate}
  \item Consider $0 < p \leq 1/(n-1)$. For the values of $p$ in this interval, the valid range of $k$ in case (c) corresponds to integer values of $k \geq 1+(n-1)p$ which means $k \geq 2$. For the probability $0 < p \leq 1/n$, the index $i$ satisfies:\begin{equation*}
\begin{array}{lll}
i
& = & \displaystyle  \left\lceil np\bigg(1-\frac{1-p}{k-np}\bigg)\right\rceil\\
& = & 1\\
& & [\mbox{since } 0< np \leq 1   \mbox{ and } 1-p \in (0,1)  \mbox{ and } k-np > 1-p].
    \end{array}
 \end{equation*}
 For the probability $1/n < p \leq 1/(n-1)$, the index $i$ satisfies:
\begin{equation*}
\begin{array}{lll}
i  &=&\displaystyle \left\lceil(n-1)p\left(\frac{\frac{k-1}{n-1}-p}{\frac{k}{n}-p}\right)\right\rceil\\
  & = & 1\\
& & [\mbox{since } 0< (n-1)p \leq 1   \mbox{ and } 0<\frac{k-1}{n-1}-p \leq \frac{k}{n}-p ].\\
    \end{array}
 \end{equation*}
Hence, the bound in \eqref{eq:SSSH} is tight in this case for all integer values of $k \geq 2$.
  \item For $p > 1/(n-1)$, the index $i = 1$ when $k(np-1) \leq n(n-1)p^2$. This corresponds to all integer values $k \in [\left\lceil 1+(n-1)p \right\rceil, \left\lfloor n(n-1)p^2/(np-1) \right\rfloor]$.
 \end{enumerate}
\end{proof}
A specific instance to show the tightness of the Chebyshev bound is to set $p = 1/2$, $k = n$ and $n = 2^m-1$ where $m$ is an integer. Using $m$ independent Bernoulli random variables it is then possible to construct $n$ pairwise independent Bernoulli random variables (see Tao \cite{tao}, Goemans \cite{goemans2015}, Pass and Spektor \cite{passbrendan2018} for this construction). Proposition \ref{prop:tighthom}(a) includes this instance (set $\alpha = (n-1)/2$, $k = n$ and $n = 2^m-1$). In addition, Proposition \ref{prop:tighthom}(a) identifies other values of $p$ and $k$ where the Chebyshev bound is tight. Proposition \ref{prop:tighthom}(b) also shows that the Schmidt, Siegel and Srinivasan bound is tight for identical marginals for small probability values ($p \leq 1/(n-1)$), for all values of $k$, except $k = 1$. We now provide a numerical illustration of the results in Theorem \ref{thm:idenBPtight} and Proposition \ref{prop:tighthom}.
\begin{example} [Identical marginals] \
In Table \ref{table1}, we provide a numerical comparison of the bounds for $n = 11$ for a set of values of $p$ and $k$. The instances in Table \ref{table1} cover all the conditions identified in Proposition \ref{prop:tighthom} when the Chebyshev and Schmidt, Siegel and Srinivasan bounds are tight. The instances when the Chebyshev bound is tight correspond to (i) $p = 0.1$ (here $\alpha = 1$ and the Chebyshev bound is tight for $k = 2$ and $k = 11$), (ii) $p = 0.2$ (here $\alpha = 2$ and the Chebyshev bound is tight for $k = 3$ and $k = 11$) and (iii) $p = 0.5$ (here $\alpha = 5$ and the Chebyshev bound is tight for $k = 6$ and $k = 11$).
The Schmidt, Siegel and Srinivasan bound is tight for the small values of $p = 0.01, 0.05, 0.10$ (which are less than or equal to $1/(n-1) = 0.1$) and for all values of $k$, except $k=1$.
\begin{table}[H]
\vspace{-0.45cm}
\footnotesize
\caption{Upper bound on probability of sum of random variables for $n = 11$. For each value of $p$ and $k$, the table provides the tight bound in \eqref{BP>=iden} followed by the Chebyshev bound \eqref{eq:ChebyH} and the Schmidt, Siegel and Srinivasan bound \eqref{eq:SSSH}. The underlined instances illustrate nontrivial cases when the upper bounds in either \eqref{eq:ChebyH} or \eqref{eq:SSSH} are tight.}\label{table1}
\begin{center}
\scalebox{0.79}{$
\footnotesize{\begin{tabular}
{|l|l|l|l|l|l|l|l|l|l|l|l|}
  \hline
p/k& 1 & 2 & 3 & 4 & 5 & 6 & 7 & 8 & 9 & 10 & 11 \\ \hline
0.01 & 0.1090 & 0.00550 & 0.00184 &  0.00092 & 0.00055 & 0.00037 & 0.00027 & 0.00020 & 0.00016 & 0.00013 & 0.00010 \\
 & 0.1208 & 0.02959  & 0.01288 & 0.00715 & 0.00454 & 0.00313 &  0.00229 &  0.00175 &  0.00138 &  0.00112 & 0.00092\\  & 0.11000 & \underline{0.00550} & \underline{0.00184} & \underline{0.00092}   & \underline{0.00055} & \underline{0.00037} & \underline{0.00027} & \underline{0.00020} & \underline{0.00016} & \underline{0.00013} & \underline{0.00010}\\
 \hline
0.05 & 0.5250  & 0.13750 & 0.04583 & 0.02292 & 0.01375 & 0.00917 & 0.00655 & 0.00491 & 0.00382  & 0.00306 & 0.00250\\
& 0.7206 & 0.19905  & 0.08008 &0.04205 & 0.02571 &0.01729 & 0.01240  & 0.00933 & 0.00726 & 0.00582 & 0.00477\\
& 0.5500 & \underline{0.13750} & \underline{0.04583} & \underline{0.02292} &  \underline{0.01375} & \underline{0.00917} &  \underline{0.00655} & \underline{0.00491} & \underline{0.00382} & \underline{0.00306} & \underline{0.00250} \\ \hline
0.10 & 1 & 0.55000 & 0.18333 & 0.09167 &0.05500 & 0.03667 & 0.02620 &  0.01965 & 0.01528 & 0.01223 &  0.01000\\
& 1 & \underline{0.55000} & 0.21522 & 0.10532 & 0.06112 & 0.03960  &0.02766 &0.02038 &0.01562& 0.01235&\underline{0.01000}\\
& 1 & \underline{0.55000} & \underline{0.18333} & \underline{0.09167} & \underline{0.05500} & \underline{0.03667} & \underline{0.02620} &  \underline{0.01965} & \underline{0.01528} & \underline{0.01223} &  \underline{0.01000}\\\hline
0.11 &  1 & 0.59950 & 0.22184 & 0.11092 & 0.06655 & 0.04437 & 0.03037 &  0.02170 &  0.01627   & 0.01266 & 0.01013 \\
 & 1 & 0.63310 & 0.25156 & 0.12154  &  0.06975 & 0.04484 & 0.03113 &  0.02283 &0.01744 & 0.01375 & 0.01112\\
 & 1 & 0.60500 & \underline{0.22184} & \underline{0.11092} & \underline{0.06655} & \underline{0.04437} & 0.03170 & 0.02377 & 0.01849 &  0.01479 & 0.01210\\\hline
0.15 & 1 & 0.78750  & 0.41250   & 0.19584 & 0.09792     &0.05875  & 0.03916  &0.02798  & 0.02098  &0.01632  &0.01306    \\
      & 1 & 0.91968 &  0.43489 & 0.20253 &   0.11109 &0.06901 &  0.04672 & 0.03362 & 0.02531 &   0.01972 &   0.01579 \\
     & 1   &0.82500 & \underline{0.41250} & 0.20625  & 0.12375  & 0.08250   &0.05893  &  0.04419 & 0.03437  & 0.02750 & 0.02250    \\\hline
0.20 & 1 &  1 & 0.73334  &  0.33334  &  0.16667  &  0.10000 &  0.06667   & 0.04762 & 0.03572   &0.02778    &  0.02223          \\
&      1 & 1 &  \underline{0.73334}  & 0.35200 & 0.18334 &   0.10865 & 0.07097  &0.04972 & 0.03667 & 0.02812 & \underline{0.02223}\\
&        1  & 1  &  \underline{0.73334} &   0.36667 & 0.22000 &   0.14667  & 0.10477  & 0.07858   &0.06112  & 0.04889   & 0.04000              \\
\hline
0.50& 1 & 1 & 1 & 1 & 1 & {0.91667} & 0.54167 &0.29167 &0.17500 & 0.11667  &  0.08334\\
 & 1 & 1 & 1 & 1 & 1 & \underline{0.91667} & 0.55000 &0.30556 & 0.18334 & 0.11957 & \underline{0.08334}\\
 & 1 & 1 & 1 & 1 & 1 & \underline{0.91667} & 0.65477 & 0.49108 & 0.38195 & 0.30556 & 0.25000 \\ \hline
\end{tabular}
}$}
\end{center}
\end{table}
 It is also clear why the Schmidt, Siegel and Srinivasan bound is not tight for $k = 1$, since it just reduces to the Markov bound $np$ and does not exploit the pairwise independence information.
For $k = 1$, the tight bound from Theorem \ref{thm:idenBPtight} is given by $np-(n-1)p^2$ (see Theorem \ref{thm:unionbound} which reduces to the same bound for $k = 1$).
For larger values of $p$ above $0.1$, such as $p = 0.11$ in the table, from Proposition \ref{prop:tighthom}(b), the Schmidt, Siegel and Srinivasan bound is tight for $k \in [\left\lceil 2.1\right\rceil, \left\lfloor 6.33 \right\rfloor]$ which corresponds to $k \in [3,6]$. This can be similarly verified for the other probabilities $p = 0.15, 0.2, 0.5$ in the table.
\end{example}

 \subsection{Tightness of ordered bounds in a special case}\label{subsec:almostidentight}
 In this section, we provide an instance when two of the ordered bounds derived in Section \ref{sec:nonidenticalnewbounds} are shown to be tight. While the ordered bounds in Theorem \ref{thm:nonidentical} are not tight in general, the next proposition identifies a special case with almost identical marginals when the bounds of Schmidt, Siegel and Srinivasan in \eqref{POSB1} and Boros and Pr\'{e}kopa in \eqref{Prekopa3} are shown to be attained.
\begin{prop}\label{prop:tightnonid}\
Suppose the marginal probabilities equal $p \in (0,1/(n-1)]$ for $n-1$ random variables and $q \in (0,1)$ for one random variable.  Then, the ordered bounds in (\ref{POSB1}) and (\ref{Prekopa3}) are tight for the following three instances and are given by:
\begin{equation}
\overline{P}(n,k,p,q) = \left\{
\begin{array}{llll}   \label{POSBalidenthm}
\displaystyle \frac{\binom{n-1}{2}p^2}{\binom{k-1}{2}}, & k \geq 3,  \; q \geq (n-2)p,  & \mbox{(a)},\\
\displaystyle \frac{\binom{n-1}{2}p^2}{\binom{k-1}{2}}, & k \in \big[\ceil{2+ (n-2)p/q},\; n\big] ,\; p \leq q < (n-2)p, & \mbox{(b)}, \\
\displaystyle pq, & k=n, \; 0 < q < p, & \mbox{(c)}.
\end{array}
\right.
\end{equation}
\end{prop}
\begin{proof}
 We first prove that the ordered bounds of Schmidt, Siegel and Srinivasan and Boros and Pr\'{e}kopa reduce to the bound in \eqref{POSBalidenthm} in each of the three cases and then show that the bound is tight.
 \texitem{(1)} Show reduction of ordered bounds to the bound in (\ref{POSBalidenthm}): Let $\overline{P}(n,k,p,q)$ represent the tightest upper bound when $n-1$ probabilities are $p$ and one is $q$. It can be observed that the bound in \eqref{POSBalidenthm} is non-trivial for the three instances since:
 \begin{equation*}
\begin{array}{rllll}   \label{nontrivial}
\displaystyle \frac{\binom{n-1}{2}p^2}{\binom{k-1}{2}}&=&  \displaystyle \frac{(n-1)p(n-2)p}{(k-1)(k-2)}< 1,\\
   & & [\mbox{since}\;\displaystyle (n-2)p < (n-1)p \leq 1\; \mbox{and}\;k \geq 3 \; \mbox{for cases (a) and (b)}],\\
   pq & < & 1,\\
   & & [\mbox{since}\;q<p<1 \;\mbox{for case (c)}].
\end{array}
\end{equation*}
It is easy to verify that the ordered Schmidt, Siegel and Srinivasan bound in \eqref{POSB1} reduces to the bound in \eqref{POSBalidenthm} for a specific parameter $r_2$ in each of the three cases:
\begin{equation}
\begin{array}{llll}   \label{SSSreduction}
 r_2=1,&  \mbox{cases (a) and (b)},\\
r_2=n-2, &  \mbox{case (c)}.
\end{array}
\end{equation}
It can be similarly verified that the ordered Boros and Pr\'{e}kopa bound in \eqref{Prekopa3} reduces to the bound in \eqref{POSBalidenthm} with the following parameters $r$ and $i$ in each of the three cases:
\begin{equation}
\begin{array}{llll}   \label{POSBalidenthm2}
 r=1,\;i=0, &  \mbox{cases (a) and (b)},\\
r=n-2,\;i=0, &  \mbox{case (c)}.
\end{array}
\end{equation}
The effectiveness of ordering is demonstrated by \eqref{SSSreduction} and \eqref{POSBalidenthm2} in that the ordered bounds of Schmidt, Siegel and Srinivasan and  Boros and Pr\'{e}kopa correspond to $r > 0$ while their unordered counterparts in \eqref{SSS} and \eqref{BorosPrekopa} correspond to $r=0$ (considering all $n$ variables). The unordered bounds are thus strictly weaker than the ordered bounds which in turn are tight as proved in the next step.
 \texitem{(2)} Prove tightness of the bound in (\ref{POSBalidenthm}) by constructing extremal distributions: Consider the linear program to compute $\overline{P}(n,k,p,q) $ which can be written as:
\begin{equation}
\begin{array}{rlllll} \label{large_pw}
\displaystyle \overline{P}(n,k,p,q)  = \max & \displaystyle \sum_{\mbs{c} \in \{0,1\}^n: \sum_{t} {c}_{t} \geq k} \theta(\mb{c}) \\
 \mbox{s.t}  &\displaystyle  \sum_{\mbs{c} \in \{0,1\}^n} \theta(\mb{c})= 1,\\
 & \displaystyle \sum_{\mbs{c} \in \{0,1\}^n: c_i = 1} \theta(\mb{c})= p , & \forall i \in [n-1],\\
  & \displaystyle \sum_{\mbs{c} \in \{0,1\}^n: c_n = 1} \theta(\mb{c})= q, \\
 & \displaystyle \sum_{\mbs{c} \in \{0,1\}^n: c_i = 1, c_j = 1}\theta(\mb{c})= p^2, & \forall (i,j) \in K_{n-1}, \\
 & \displaystyle \sum_{\mbs{c} \in \{0,1\}^n: c_i = 1, c_n = 1}\theta(\mb{c})= pq, & \forall i \in [n-1],\\
  & \displaystyle \theta(\mb{c}) \geq 0, & \forall \mb{c} \in \{0,1\}^n.
\end{array}
\end{equation}
We now proceed to prove tightness of the bound in \eqref{POSBalidenthm} for each of the three instances of the $(n,k,p,q)$ tuple by constructing feasible distributions of \eqref{large_pw} which attain the bound.
\begin{enumerate}
\item ${\overline{P}(n,k,p,q) =\displaystyle \frac{\binom{n-1}{2}p^2}{\binom{k-1}{2}}}$ (cases (a) and (b)):\\
 The following distribution attains the tight bound:
\begin{equation} \label{dalmost1}
\begin{array}{llllll}
\displaystyle \theta(\mb{c}) = \\
\left\{\begin{array}{llr}
(1-q)(1-(n-1)p), & \mbox{if }  \displaystyle \sum_{t \in [n]} {c}_{t} =0,  & (x),\\
p(1-q), & \mbox{if}  \displaystyle \sum_{t \in [n-1]} {c}_{t} =1, c_n =0, & (y), \\
q(1-(n-1)p)+\frac{(n-1)(n-2)p^2}{(k-1)}, & \mbox{if }  \displaystyle\sum_{t \in [n-1]} {c}_{t} =0, c_n =1,  & (z),\\
p(q-\frac{n-2}{k-2} p), & \mbox{if }  \displaystyle \sum_{t \in [n-1]} {c}_{t} =1, c_n = 1,  & (u),\\
\frac{p^2}{\binom{n-3}{k-3}}, & \mbox{if } \displaystyle \sum_{t \in [n-1]} {c}_{t} = k-1, c_n =1,  & (v).
\end{array}\right.
\end{array}
\end{equation}
We use symbols $x,y,z,u,v$ to denote the probability of the associated scenarios in (\ref{dalmost1}). The constraints in \eqref{large_pw} reduce to:
\[\begin{array}{ll} \label{constraints1}
 \binom{n-2}{k-2}  v+u+y=p&   \\
 \binom{n-1}{k-1}   v+(n-1)u+z= q&   \\
 \binom{n-3}{k-3}v = p^2& \\
 \binom{n-2}{k-2}v + u=pq &  \\
 x+y+z+u+v=1,&
\end{array}
\]
and using $x,y,z,u,v$ from \eqref{dalmost1}, it can be easily verified that all of the above  constraints are satisfied.  The non-negativity constraints for $y,v$ are satisfied while  $x \geq 0, \; z\geq 0$ is satisfied since $(n-1)p \leq 1$. Remaining case is $u$, for which we have:
\begin{equation*}
\begin{array}{lllll} \label{nonneg1}
\textrm{case (a):} & \displaystyle u & = & p(q-\frac{n-2}{k-2} p) \\
&&\geq &  p(q-\frac{n-2}{3-2} p) \\
  & && [\mbox{since}\;  k \geq 3 ] \\
&&= & p(q-(n-2) p) \\
  & && [\mbox{since}\;  q >(n-2)p ]  \\
  & &\geq & 0  \\
  \textrm{case (b):} & \displaystyle u & = & p(q-\frac{n-2}{k-2} p) \\
&& \geq &  p(q-\frac{k-2}{k-2} q)\\
 & &&  [\mbox{since } k\geq 2+ (n-2)p/q ]   \\
 && =& 0.
\end{array}
\end{equation*}
The only support points contributing to the objective function are the first set of $ \binom{n-1}{k-1}$ scenarios, and so we have $\overline{P}(n,k,p,q) =  \binom{n-1}{k-1} {p^2}/{\binom{n-3}{k-3}}  =  {\binom{n-1}{2}p^2}/{\binom{k-1}{2}}$.
\item
${\overline{P}(n,k,p,q) =pq}$ (case (c)): \\ The following distribution attains the tight bound $\displaystyle pq$:
\begin{equation} \label{dalmost2}
\displaystyle \theta(\mb{c}) = \left\{\begin{array}{llr}
(1-p)(1-(n-2)p-q), & \mbox{if } \displaystyle \sum_{t \in [n]} {c}_{t} =0,  & (x),\\
p(1-p), & \mbox{if}  \displaystyle\sum_{t \in [n-1]} {c}_{t} =1, c_n =0,   & (y),\\
q(1-p),& \mbox{if }  \displaystyle\sum_{t \in [n-1]} {c}_{t} =0, c_n =1,  & (z),\\
p(p-q), & \mbox{if }  \displaystyle \sum_{t \in [n-1]} {c}_{t} =n-1, c_n =0,  & (u),\\
pq, & \mbox{if } \displaystyle \sum_{t \in [n]} {c}_{t} = n, & (v). \\
\end{array}\right.
\end{equation}
The constraints in \eqref{large_pw} reduce to:
\[\begin{array}{ll} \label{constraints3}
 \displaystyle y+u+v=p&   \\
 \displaystyle z+v= q&  \\
 \displaystyle u+v = p^2&  \\
 \displaystyle v=pq &   \\
 \displaystyle x+y+z+u+v=1,&
\end{array}
\]
and using $x,y,z,u,v$ from \eqref{dalmost2}, it can be easily verified that all of the above  constraints are satisfied. The non-negativity contraints for $y, z,u,v$ are satisfied by $0< q \leq p \leq 1$ while for $x$, we have:
\begin{equation*}
\begin{array}{lll} \label{nonneg3}
 \displaystyle x &=&(1-p)(1-(n-2)p-q) \\
 & \geq &  \displaystyle (1-p)(1-(n-2)p-p)   \\
  &  & [\mbox{since}\;  q < p ]   \\
&= &  \displaystyle (1-p)(1-(n-1)p)    \\
& \geq & 0   \\
 &  & [\mbox{since}\;  (n-1)p \leq 1].
\end{array}
\end{equation*}
The distribution in \eqref{dalmost2} attains the bound $pq$.
\end{enumerate}
We have thus constructed two feasible probability distributions in \eqref{dalmost1} and \eqref{dalmost2} which attain the bound in \eqref{POSBalidenthm} in each of the three instances defined by the $(n,k,p,q)$ tuple. Hence the parameters $r_2,\;r$ in \eqref{SSSreduction} and \eqref{POSBalidenthm2} defined for each of the three cases must be the minimizers which exactly reduce the ordered bounds in \eqref{POSB1} and \eqref{Prekopa3} to the tight bound in \eqref{POSBalidenthm}.
\end{proof}

\begin{example}
This example demonstrates the usefulness of Proposition \ref{prop:tightnonid} when $n=100$ and $p=0.01$ where $(n-1)p\leq 1$. It compares the tight bounds computed from \eqref{POSBalidenthm} with the unordered bounds of Schmidt, Siegel and Srinivasan from \eqref{SSS} and that of Boros and Pr\'{e}kopa from \eqref{BorosPrekopa}.
\begin{figure}[htbp]
\begin{subfigure}[b]{0.45\linewidth}
\includegraphics[scale=0.4]{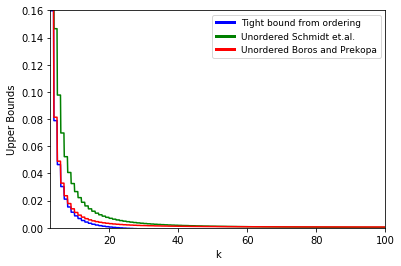}
\caption{$q=0.99,\;q \geq (n-2)p,\;k\geq 3$ }
\label{fig:alhomo1a}
\end{subfigure}
\begin{subfigure}[b]{0.45\linewidth}
\includegraphics[scale=0.4]{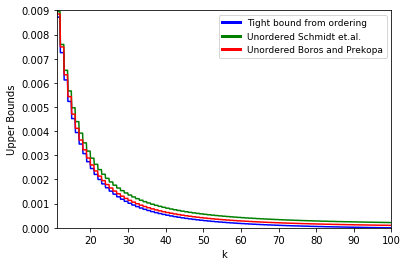}
\caption{$q=0.1,\;p\leq q < (n-2)p,\;k\geq 12$ }
\label{fig:alhomo1b}
\end{subfigure}%
\caption{Comparison of unordered bounds with tight bound when $n=100,\; p=0.01$}
\label{fig:alhomo1c}
\end{figure}

Figure \ref{fig:alhomo1a} plots the two unordered bounds along with the tight bound when $ q=0.99$ (case (a) of Proposition \ref{prop:tightnonid}), where the tight bound is valid for all $k$ in $[3,n]$, while Figure \ref{fig:alhomo1b} compares the bounds when $q=0.1$ (case (b) of Proposition \ref{prop:tightnonid}) for $k\geq 12$ as the tight bound is valid when $k\geq \ceil{2+(n-2)p/q}=\ceil{11.8}=12$. The unordered Boros and Pr\'{e}kopa bound is much tighter than the unordered Schmidt, Siegel and Srinivasan bound in both figures.
Hence, Figure \ref{fig:alhomo1c} demonstrates that with ordering, the relative improvement of the Schmidt, Siegel and Srinivasan bound is much better than that of the Boros and Pr\'{e}kopa bound although both the ordered bounds reduce to the tight bound in \eqref{POSBalidenthm}.
\end{example}

\section{Conclusion} \label{sec:conclusion}
In this paper we have provided results towards finding tight probability bounds for the sum of $n$ pairwise independent random variables adding up to at least an integer $k$. In Section \ref{sec:unionbound}, we first established with Lemma \ref{lem:bivarfeas} that a feasible correlated distribution of a Bernoulli random vector $\tilde{\mb{c}}$ with an arbitrary univariate probability vector $\mb{p} \in [0,1]^n$ and transformed bivariate probabilities $p_ip_j/p$ where $\max_i p_i \leq p \leq 1$, always exists (this result was then extended to prove the existence of an alternate correlated Bernoulli random vector in Corollary \ref{cor:bivarfeasvariant}). Theorem \ref{thm:unionbound} then established that with pairwise independence, the Hunter \cite{hunter1976} and Worsley \cite{worsley1982} bound is tight for any $\mb{p} \in [0,1]^n$, which, to the best of our knowledge, has not been shown thus far in the literature dedicated to this topic. In fact, paraphrasing from Boros \cite{boros2014} (Section 1.2), ``As far as we know, in spite of the several studies dedicated to this problem, the complexity status of this problem, for feasible input, seems to be still open even for bivariate probabilities''. With pairwise independent random variables, feasibility is guaranteed and Theorem \ref{thm:unionbound} shows that the tightest upper bound is computable in polynomial time (in fact in a simple closed-form), thus providing a partial positive answer towards this question. The proof included the explicit construction of an extremal distribution (though not unique) in Table \ref{table:hwprobdist}, that attains this bound. We then showed in Proposition \ref{prop:25} that the ratio of the Boole union bound and the pairwise independent bound is upper bounded by $4/3$ and that this bound is attained. Applications of the result in correlation gap analysis and bottleneck optimization (in the distributionally robust optimization context) were discussed in examples \ref{ex:correlationanalysis} and \ref{ex:bottleneck}. The tight upper bound on the union probability was then used to derive a closed-form expression for the tight lower bound on the intersection probability in Corollary \ref{cor:intersection}, which, to the best of our knowledge, appears to be unknown in the literature. In Section \ref{sec:nonidenticalnewbounds}, for $k \geq 2$, we proposed new bounds exploiting ordering of the probabilities (which are at least as good as the unordered bounds) and argued that the ordered Boros and Pr\'{e}kopa bound must be at least as good as the other two ordered bounds proposed in Theorem \ref{thm:nonidentical}. To the best of our knowledge, this idea of ordering has not been exploited thus far to tighten probability bounds for pairwise independent random variables.
We then showed in Section \ref{subsec:furthertight} that the ordered bounds can be further tightened by using the tight bound for $k = 1$ from Theorem \ref{thm:unionbound}. Numerical examples in Section \ref{subsec:numericals_noniden} then demonstrated that while the Boros and Pr\'{e}kopa bound is uniformly the best performing of the three ordered bounds, the Schmidt, Siegel and Srinivasan bound shows the best improvement with ordering, in the examples considered.
Section \ref{sec:tightinstances} provided instances when the unordered and ordered bounds are tight.
In Section \ref{subsec:idenBP}, for the special case of identical probabilities $p \in [0,1]$ and any $k \in [n]$, we used a constructive proof exploiting the symmetry in the problem, to identify the best upper bound $\overline{P}(n,k,p)$ in closed-form and a corresponding extremal distribution. This result was further extended to provide tight bounds (not necessarily closed-form) for more general $t$-wise independent identical variables in Corollary \ref{cor:twiseiden}.  We then demonstrated the usefulness of this result by identifying instances when the existing unordered bounds are tight. Section \ref{subsec:almostidentight} demonstrated the usefulness of the ordered bounds by identifying an instance with $n-1$ identical probabilities (along with additional conditions on the identical probability and $k$), when the ordered bounds are tight.

\noindent We believe several interesting research questions arise from this work, two of which we list below:
\texitem{(a)} To the best of our knowledge, the computational complexity of evaluating (or approximating) the bound $\overline{P}(n,k,\mb{p})$ for general $n$, $k$ and $\mb{p} \in [0,1]^n$ is still unresolved. While we provide the answer in closed-form for $k = 1$, a natural question that arises is whether the tight bounds for general $k \geq 2$ with pairwise independent random variables are efficiently computable (or efficient to approximate)? We leave this for future research.
\texitem{(b)} The upper bound of $4/3$ in Section \ref{subsec:43bound} is derived for the ratio between the maximum probability for the union of arbitrarily dependent events and the probability of the union of pairwise independent events. We conjecture this upper bound is valid for the expected value of all non-decreasing, nonnegative submodular functions (of which the probability of the union is a special case) and leave it as an open question.

\section*{Acknowledgments}
We would like to thank the Associate Editor Prasad Tetali and the reviewers for their careful reading of the paper and useful inputs.

\bibliographystyle{siamplain}
\bibliography{references}

\begin{thebibliography}{10}

\bibitem{Agrawal2012}
{\sc S.~Agrawal, Y.~Ding, A.~Saberi, and Y.~Ye}, {\em Price of correlations in
  stochastic optimization}, Operations Research, 1 (2012), pp.~150--162.

\bibitem{Babai}
{\sc L.~Babai}, {\em Entropy versus pairwise independence}, Available at
  http://people.cs.uchicago.edu/$\sim$ laci/papers/13augEntropy.pdf,  (2013).

\bibitem{benjamini12}
{\sc I.~Benjamini, O.~Gurel-Gurevich, and R.~Peled}, {\em {On k-wise
  independent distributions and boolean functions}}, {Working Paper, In: arXiv
  preprint:1201.3261},  (2012).

\bibitem{bernstein1946}
{\sc S.~Bernstein}, {\em {Theory of probability, Moscow-Leningrad}},  (1946).

\bibitem{boole1854}
{\sc G.~Boole}, {\em {The Laws of Thought (1916 reprint)}}, 1854.

\bibitem{boros1989}
{\sc E.~Boros and A.~Pr\'{e}kopa}, {\em {Closed form two-sided bounds for
  probabilities that at least r and exactly r out of n events occur}},
  {Mathematics of Operations Research}, 14 (1989), pp.~317--342.

\bibitem{boros2014}
{\sc E.~Boros, A.~Scozzari, F.~Tardella, and P.~Veneziani}, {\em {Polynomially
  computable bounds for the probability of the union of events}}, {Mathematics
  of Operations Research}, 39 (2014), pp.~1311--1329.

\bibitem{Calinescu2}
{\sc G.~Calinescu, C.~Chekuri, M.~P\'{a}l, and J.~Vondr\'{a}k}, {\em Maximizing
  a monotone submodular function subject to a matroid constraint}, SIAM Journal
  on Computing, 40 (2007), pp.~1740--1766.

\bibitem{Rao}
{\sc N.~R. Chaganty and H.~Joe}, {\em Range of correlation matrices for
  dependent bernoulli random variables}, Biometrika, 1 (2006), pp.~197--206.

\bibitem{chebyshev1867}
{\sc P.~Chebyshev}, {\em {Des valeurs moyennes}}, {Journal de Mathématiques
  Pures et Appliquées}, 2 (1867), pp.~177--184.

\bibitem{chernoff}
{\sc H.~Chernoff}, {\em {A measure of asymptotic efficiency for tests of a
  hypothesis based on the sum of observations}}, {Annals of Mathematical
  Statistics}, 23 (1952), pp.~493--509.

\bibitem{dawson1967}
{\sc D.~A. Dawson and D.~Sankoff}, {\em {An inequality for probabilities}},
  {Proceedings of the American Mathematical Society}, 18 (1967), pp.~504--507.

\bibitem{decaen1997}
{\sc D.~de~Caen}, {\em {A lower bound on the probability of a union}},
  {Discrete Mathematics}, 169 (1997), pp.~217--220.

\bibitem{dohmen2007}
{\sc K.~Dohmen and P.~Tittmann}, {\em {Improved Bonferroni inequalities and
  binomially bounded functions}}, {Electronic Notes in Discrete Mathematics},
  28 (2007), pp.~91--93.

\bibitem{Edmonds}
{\sc J.~Edmonds and D.~R. Fulkerson}, {\em Bottleneck extrema}, Journal of
  Combinatorial Theory, 3 (1970), pp.~299--306.

\bibitem{Emrich}
{\sc L.~J. Emrich and M.~R. Piedmonte}, {\em A method for generating
  high-dimensional multivariate binary variates}, The American Statistician, 45
  (1991), pp.~302--304.

\bibitem{feller1959}
{\sc W.~Feller}, {\em {Non-Markovian processes with the semigroup property}},
  {The Annals of Mathematical Statistics}, 30 (1959), pp.~1252--1253.

\bibitem{feller}
{\sc W.~Feller}, {\em {An Introduction to Probability Theory and Its
  Applications: Volume I}}, Wiley Series in Probability and Mathematical
  Statistics, 3~ed., 1968.

\bibitem{frechet1935}
{\sc M.~Fr{\'e}chet}, {\em {G{\'e}n{\'e}ralisation du th{\'e}oreme des
  probabilit{\'e}s totales}}, {Fundamenta mMthematicae}, 1 (1935),
  pp.~379--387.

\bibitem{galambos1975}
{\sc J.~Galambos}, {\em {Methods for proving Bonferroni type inequalities}},
  {Journal of the London Mathematical Society}, 2 (1975), pp.~561--564.

\bibitem{galambos1977}
{\sc J.~Galambos}, {\em {Bonferroni inequalities}}, {The Annals of
  Probability},  (1977), pp.~577--581.

\bibitem{garnett}
{\sc B.~Garnett}, {\em {Small deviations of sums of independent random
  variables}}, {Journal of Combinatorial Theory, Series A}, 169 (2020),
  pp.~105--119.

\bibitem{Gavinsky}
{\sc D.~Gavinsky and P.~Pudl\'{a}k}, {\em On the joint entropy of
  d-wise-independent variables}, Commentationes Mathematicae Universitatis
  Carolinae,  (2016), pp.~333--343.

\bibitem{geisser1962}
{\sc S.~Geisser and N.~Mantel}, {\em {Pairwise independence of jointly
  dependent variables}}, {The Annals of Mathematical Statistics}, 33 (1962),
  pp.~290--291.

\bibitem{goemans2015}
{\sc M.~Goemans}, {\em {Chernoff bounds, and some applications}}, {Lecture
  Notes, MIT},  (2015).

\bibitem{hailperin}
{\sc T.~Hailperin}, {\em {Best possible inequalities for the probability of a
  logical function of events}}, {The American Mathematical Monthly}, 72 (1965),
  pp.~343--359.

\bibitem{hoeffding}
{\sc W.~Hoeffding}, {\em {Probability inequalities for sums of bounded random
  variables}}, {Journal of the American Statistical Association}, 58 (1963),
  pp.~13--30.

\bibitem{hunter1976}
{\sc D.~Hunter}, {\em {An upper bound for the probability of a union}},
  {Journal of Applied Probability}, 13 (1976), pp.~597--603.

\bibitem{joffe1974}
{\sc A.~Joffe}, {\em {On a set of almost deterministic $ k $-independent random
  variables}}, {The Annals of Probability}, 2 (1974), pp.~161--162.

\bibitem{karloffmansour1994}
{\sc H.~Karloff and Y.~Mansour}, {\em {On construction of k-wise independent
  random variables}}, in {Proceedings of the 26th Annual ACM Symposium on
  Theory of Computing}, 1994, pp.~564--573.

\bibitem{koller}
{\sc D.~Koller and N.~Meggido}, {\em {Construcing small sample spaces
  satisfying given constraints}}, {SIAM Journal on Discrete Mathematics}, 7
  (1994), pp.~260--274.

\bibitem{kounias1968}
{\sc E.~G. Kounias}, {\em {Bounds for the probability of a union, with
  applications}}, {The Annals of Mathematical Statistics}, 39 (1968),
  pp.~2154--2158.

\bibitem{kounias1976}
{\sc S.~Kounias and J.~Marin}, {\em {Best linear Bonferroni bounds}}, {SIAM
  Journal on Applied Mathematics}, 30 (1976), pp.~307--323.

\bibitem{tahakara2000}
{\sc H.~Kuai, F.~Alajaji, and G.~Takahara}, {\em {A lower bound on the
  probability of a finite union of events}}, {Discrete Mathematics}, 215
  (2000), pp.~147--158.

\bibitem{kwerelstringent1975}
{\sc S.~M. Kwerel}, {\em {Most stringent bounds on aggregated probabilities of
  partially specified dependent probability systems}}, {Journal of the American
  Statistical Association}, 70 (1975b), pp.~472--479.

\bibitem{lancaster1965}
{\sc H.~O. Lancaster}, {\em {Pairwise statistical independence}}, {The Annals
  of Mathematical Statistics}, 36 (1965), pp.~1313--1317.

\bibitem{luby}
{\sc M.~Luby and A.~Widgerson}, {\em {Pairwise independence and
  derandomization}}, { Foundations and Trends in Theoretical Computer Science},
  1 (2005), pp.~239--201.

\bibitem{Lunn}
{\sc A.~D. Lunn and S.~J. Davies}, {\em A note on generating correlated binary
  variables}, Biometrika, 85 (1998), pp.~487--490.

\bibitem{maurer}
{\sc W.~Maurer}, {\em {Bivalent trees and forests or upper bounds for the
  probability of a union revisited}}, {Discrete Applied Mathematics}, 6 (1983),
  pp.~157--171.

\bibitem{mori1985}
{\sc T.~F. M{\'o}ri and J.~G. Sz{\'e}kely}, {\em {A note on the background of
  several Bonferroni--Galambos-type inequalities}}, {Journal of Applied
  Probability}, 22 (1985), pp.~836--843.

\bibitem{brien1980}
{\sc G.~L. O'Brien}, {\em {Pairwise independent random variables}}, {The Annals
  of Probability}, 8 (1980), pp.~170--175.

\bibitem{passbrendan2018}
{\sc B.~Pass and S.~Spektor}, {\em {On Khintchine type inequalities for k-wise
  independent Rademacher random variables}}, {Statistics \& Probability
  Letters}, 132 (2018), pp.~35--39.

\bibitem{peled}
{\sc R.~Peled, A.~Yadin, and A.~Yehudayoff}, {\em {The maximal probability that
  k-wise independent bits are all 1}}, {Random Structures \& Algorithms}, 38
  (2011), pp.~502--525.

\bibitem{pinelis2021exact}
{\sc I.~Pinelis}, {\em Exact lower bound on an ‘exactly one’probability},
  Bulletin of the Australian Mathematical Society, 104 (2021), pp.~330--336.

\bibitem{pitowsky91}
{\sc I.~Pitowsky}, {\em Correlation polytopes: Their geometry and complexity},
  Mathematical Programming, 50 (1991), pp.~395--414.

\bibitem{platz1985}
{\sc O.~Platz}, {\em {A sharp upper probability bound for the occurrence of at
  least m out of n events}}, {Journal of Applied Probability}, 22 (1985),
  pp.~978--981.

\bibitem{prekopa1988}
{\sc A.~Pr{\'e}kopa}, {\em {Boole-Bonferroni inequalities and linear
  programming}}, {Operations Research}, 36 (1988), pp.~145--162.

\bibitem{prekopa1990}
{\sc A.~Pr{\'e}kopa}, {\em {Sharp bounds on probabilities using linear
  programming}}, {Operations Research}, 38 (1990), pp.~227--239.

\bibitem{prekopa2005}
{\sc A.~Pr{\'e}kopa and L.~Gao}, {\em {Bounding the probability of the union of
  events by aggregation and disaggregation in linear programs}}, {Discrete
  Applied Mathematics}, 145 (2005), pp.~444--454.

\bibitem{Qaqish}
{\sc B.~F. Qaqish}, {\em A family of multivariate binary distributions for
  simulating correlated binary variables with specified marginal means and
  correlations}, Biometrika, 90 (2003), pp.~455--463.

\bibitem{qiu2016}
{\sc F.~Qiu, S.~Ahmed, and S.~S. Dey}, {\em {Strengthened bounds for the
  probability of k-out-of-n events}}, {Discrete Applied Mathematics}, 198
  (2016), pp.~232--240.

\bibitem{ruger1978}
{\sc B.~R\"{u}ger}, {\em {Das maximale signifikanzniveau des Tests: ``Lehne
  $H_0$ ab, wennk untern gegebenen tests zur ablehnung f\"{u}hren''}},
  {Metrika}, 25 (1978), pp.~171--178.

\bibitem{sathe1980}
{\sc Y.~S. Sathe, M.~Pradhan, and S.~P. Shah}, {\em {Inequalities for the
  probability of the occurrence of at least m out of n events}}, {Journal of
  Applied Probability}, 17 (1980), pp.~1127--1132.

\bibitem{shrinivasan1995}
{\sc J.~Schmidt, A.~Siegel, and A.~Srinivasan}, {\em {Chernoff--Hoeffding
  bounds for applications with limited independence}}, {SIAM Journal on
  Discrete Mathematics}, 8 (1995), pp.~223--250.

\bibitem{tao}
{\sc T.~Tao}, {\em {Topics in random matrix theory}}, vol.~132, Graduate
  Studies in Mathematics, American Mathematical Society, 2012.

\bibitem{veneziani2008union}
{\sc P.~Veneziani}, {\em {Graph-based upper bounds for the probability of the
  union of events}}, {The Electronic Journal of Combinatorics}, 15 (2008).

\bibitem{veneziani2008hunter}
{\sc P.~Veneziani}, {\em {Optimality conditions for Hunter's bound}}, {Discrete
  Mathematics}, 308 (2008), pp.~6009--6014.

\bibitem{vizvari2007}
{\sc B.~Vizv{\'a}ri}, {\em {New upper bounds on the probability of events based
  on graph structures}}, {Mathematical Inequalities and Applications}, 10
  (2007), p.~217.

\bibitem{worsley1982}
{\sc K.~J. Worsley}, {\em {An improved Bonferroni inequality and
  applications}}, {Biometrika}, 69 (1982), pp.~297--302.

\bibitem{Xie}
{\sc W.~Xie, J.~Zhang, and S.~Ahmed}, {\em Distributionally robust bottleneck
  combinatorial problems: uncertainty quantification and robust decision
  making}, To appear in Mathematical Programming,  (2021).

\bibitem{yang2016}
{\sc J.~Yang, F.~Alajaji, and G.~Takahara}, {\em {Lower bounds on the
  probability of a finite union of events}}, {SIAM Journal on Discrete
  Mathematics}, 30 (2016), pp.~1437--1452.

\bibitem{yoda2016}
{\sc K.~Yoda and A.~Pr{\'e}kopa}, {\em {Improved bounds on the probability of
  the union of events some of whose intersections are empty}}, {Operations
  Research Letters}, 44 (2016), pp.~39--43.

\end{thebibliography}
\end{document}